\documentclass{amsart}
\usepackage{amsmath}
\usepackage{amsfonts}
\usepackage{amssymb}
\usepackage{amsthm}
\usepackage{indentfirst}
\title{\textbf{Graph approximations to the Laplacian spectra}}
\author{Jinpeng Lu}
\address{Jinpeng Lu: Department of Mathematics and Statistics, University of Helsinki, Helsinki 00014, Finland}
\email{jinpeng.lu@helsinki.fi}
\date{}
\numberwithin{equation}{section}
\newtheorem{bounded}{Definition}[section]
\newtheorem{geodesic}[bounded]{Definition}
\newtheorem{main}{Theorem}
\newtheorem{remark}[bounded]{Remark}
\newtheorem*{mainassumption}{Assumption}
\newtheorem{remark1}[bounded]{Remark}
\newtheorem{mirror}[bounded]{Definition}
\newtheorem{main2}[main]{Theorem}
\newtheorem{main3}[main]{Theorem}
\newtheorem*{ack}{Acknowledgement}
\newtheorem{reflected}{Lemma}[section]
\newtheorem{defW}[reflected]{Definition}
\newtheorem{W}[reflected]{Lemma}
\newtheorem{ball}[reflected]{Definition}
\newtheorem{Jacobi}[reflected]{Lemma}

\newtheorem{2.2}{Lemma}[section]
\newtheorem{2.3}[2.2]{Lemma}
\newtheorem{upperbound}[2.2]{Lemma}
\newtheorem{thetax}[2.2]{Lemma}
\newtheorem{1}[2.2]{Lemma}
\newtheorem{2}[2.2]{Lemma}
\newtheorem{lowerbound}[2.2]{Lemma}
\newtheorem{function}{Definition}[section]
\newtheorem{sobolev}[function]{Lemma}
\newtheorem{eigenvalue}[function]{Definition}
\newtheorem{eigenfunction}[function]{Proposition}

\newtheorem{convergence}{Theorem}[section]
\newtheorem{newtheta}[convergence]{Lemma}
\newtheorem{new2}[convergence]{Lemma}
\newtheorem{new3}[convergence]{Lemma}
\newtheorem{new4}[convergence]{Lemma}

\newtheorem{circle11}{Example}
\newtheorem{circle21}[circle11]{Example}

\newtheorem{torus11}[circle11]{Example}
\newtheorem{circlesegment}[circle11]{Example}

\begin{document}

\maketitle

\begin{abstract}
I prove that the spectrum of the Laplace-Beltrami operator with the Neumann boundary condition on a compact Riemannian manifold with boundary admits a fast approximation by the spectra of suitable graph Laplacians on proximity graphs on the manifold, and similar graph approximation works for metric-measure spaces glued out of compact Riemannian manifolds of the same dimension.
\end{abstract}

\renewcommand{\thefootnote}{\fnsymbol{footnote}}

\section{Introduction}

In recent years geometric methods have drawn significant attention in data analysis and machine learning. The basic premise behind these methods is that high-dimensional natural point cloud data with a small number of parameters generate a low-dimensional submanifold. In such setting the underlying manifold is typically unknown, and the common strategy is to construct a proximity graph associated with the point cloud. Since the graph approximates the underlying manifold, one naturally expects to collect information on the geometry of the manifold from the structure of the graph. The learning algorithms in this setting, referred to as manifold learning, are often developed based on the Laplace-Beltrami operator of a Riemannian manifold and the associated graph Laplacian (a finite dimensional matrix). To guarantee the convergence of such algorithms, it is necessary to prove the spectral convergence of the graph Laplacian to the Laplace-Beltrami operator.

So far the works on the spectral convergence of the graph Laplacian appear in probabilistic and non-probabilistic settings. If the sampling probability distribution is a priori known, the spectral convergence in probability for closed manifolds has been studied in \cite{BN,W,TGHS}. However, essential obstacles appear when one works towards a convergence result for manifolds with boundary. It was shown in \cite{BQWZ} that the graph Laplacian approximates a first order differential operator at points near a manifold boundary, which implies that points near the boundary have significant impact on energy estimates although amount to small volume. The spectral convergence in probability for manifolds with boundary to the Neumann Laplacian was recently proved in \cite{SW,TS} without an error estimate. While an error estimate was obtained in \cite{S}, its dependency on the underlying manifold is not explicit. More general stratified spaces with certain singularities such as intersections and corners were also considered (e.g. \cite{AA,BQWZ}), but the spectral convergence has not been established for such spaces.

On the other hand, the non-probabilistic spectral convergence has been proved in D. Burago, S. Ivanov and Y. Kurylev's work \cite{BIK} for closed manifolds, where the data points are given instead of the sampling distribution. They showed that the spectrum of the Laplace-Beltrami operator on a closed Riemannian manifold admits a fast approximation by the spectra of properly weighted graph Laplacians on the manifold as long as the graphs are dense enough. The advantage of this approach is that it does not require data points to be distributed in any regular way. Furthermore, the error estimate explicitly only depends on intrinsic geometric parameters of the underlying manifold and is independent of the sampling distribution. However, it is necessary to know a discrete measure approximating the volume of the underlying manifold locally near the data points.

The purpose of this paper is to establish the non-probabilistic spectral convergence for manifolds with boundary and metric-measure spaces with certain types of singularities, as a generalization of \cite{BIK}. As discussed earlier by \cite{BQWZ}, the graph Laplacian is dominated by a first order differential operator near manifold boundaries and certain singularities (intersections, corners). Considering that data points can cluster near boundaries or singularities to any degree if no information on the sampling distribution is given, there is little hope to prove the non-probabilistic spectral convergence of the usual graph Laplacian in these spaces. Therefore, the usual graph Laplacian needs to be modified near boundaries and singularities for this type of results. In this paper, I introduce approximations of the Laplacian spectra with the Neumann boundary condition by the spectra of suitable graph Laplacians for compact Riemannian manifolds with boundary and more generally metric-measure spaces which are glued out of compact Riemannian manifolds of the same dimension. Furthermore, the error estimates explicitly only depend on intrinsic geometric parameters of the underlying manifold.

My result implies that the closeness between the spectrum of the classical Laplacian and the spectra of graph Laplacians extends beyond manifolds. To consider a potential counterexample where no restriction on volume growth is assumed, one may find out if the convergence holds for spaces which are glued out of manifolds of different dimensions (e.g. a ball with a segment attached, where the segment is equipped with $1$-dimensional Lebesgue measure), and if the limit is closely related to the classical Laplacian. Graph Laplacians and their continuous (as opposed to discrete) analogues, named $\rho$-Laplacians in \cite{BIK2}, can be defined not only on manifold structures but also for general metric-measure spaces. It was proved in \cite{BIK2} that for a large class of metric-measure spaces, the spectra of $\rho$-Laplacians are stable under metric-measure perturbations, which implies that the convergence of the spectra of graph Laplacians is equivalent to the convergence of those of $\rho$-Laplacians. We ask under what conditions on a metric-measure space one can guarantee that graph Laplacians and $\rho$-Laplacians converge to some operator in a proper sense. On the other hand, it is possible to consider graph approximations to the Hodge Laplacian, which could provide a way to study topological invariants such as Betti numbers via operators on functions on graph structures. We will address the non-probabilistic approximations to the Riemannian connection Laplacian later in another work.\\

\begin{bounded}\label{bounded}
Suppose $M$ is an $n$-dimensional compact Riemannian manifold with smooth boundary $\partial M$ (possibly $\partial M=\emptyset$). We say $M\in \mathcal{M}_n(K_1,K_2,D,i_0)$ of dimension $n$ for some $K_1,K_2,D,i_0>0$ if the following conditions hold.\\
(1) The absolute value of sectional curvatures of $M$ and the norm of the second fundamental form of $\partial M$ embedded in $M$ are bounded by $K_1$;\\
(2) The norm of the covariant derivative of the curvature tensor of $M$ is bounded by $K_2$;\\
(3) The diameter of $M$ is bounded by $D$;\\
(4) The injectivity radius of $M$ is bounded below by $i_0$.
\end{bounded}
Recall that the injectivity radius for a manifold with boundary is defined as the minimum of the radius within which the exponential map at any point away from the boundary is a diffeomorphism, the radius within which the boundary normal coordinates exist, and the injectivity radius of the boundary. \\

For $M\in\mathcal{M}_n(K_1,K_2,D,i_0)$, we present a weighted graph structure for an arbitrary net on the manifold, which was introduced in \cite{BIK} for closed manifolds. First we fix two positive parameters $\varepsilon,\rho$ with $\varepsilon\ll\rho\ll 1$. Suppose $X_{\varepsilon}=\{x_i\}_{i=1}^N$ is a finite $\varepsilon$-net on the manifold; it forms the vertices of our graph. For the sake of simplicity, we require $x_i\notin \partial M$. Two vertices $x_i,x_j$ are connected by an edge if and only if their Riemannian distance $d(x_i,x_j)$ is smaller than $\rho$. We take any partition of the manifold into measurable subsets $V_i$ satisfying $V_i\subset B_{\varepsilon}(x_i)$, where $B_{\varepsilon}(x_i)$ is the standard geodesic ball of the manifold around $x_i$ with the radius $\varepsilon$. Then we assign the Riemannian volume of $V_i$ as the weight $\mu_i$ to each vertex $x_i$. Denote this weighted graph by $\Gamma_{\varepsilon,\rho}=\Gamma(X_{\varepsilon},\mu,\rho)$ with $\varepsilon\ll\rho\ll 1$. \\
\indent For closed manifolds and for any function $u$ on $X_{\varepsilon}$, the weighted graph Laplacian $\Delta_{\Gamma}$ is defined by
$$\Delta_{\Gamma} u(x_i)=\frac{2(n+2)}{\nu_n \rho^{n+2}}\sum_{j: d(x_i,x_j)<\rho} \mu_j(u(x_j)-u(x_i)), \textrm{ for }\partial M=\emptyset,$$ 
where $\nu_n$ is the volume of the unit ball in $\mathbb{R}^n$. This is a nonpositive self-adjoint operator with respect to the weighted discrete $L^2(X_{\varepsilon})$ norm. It was proved in \cite{BIK} that the spectra of the weighted graph Laplacians converge to the spectrum of the Laplace-Beltrami operator on a closed manifold as $\rho+\varepsilon/\rho\to 0$. The proof of the convergence relies on energy estimates which are optimal for closed manifolds. The presence of a manifold boundary or any singularity will prevent such estimates to achieve the required optimum, which motivates us to impose a symmetry assumption (to be explained later) on general metric-measure spaces. For manifold boundaries, this issue can be resolved by intuitively doubling the manifold.\\
\indent For manifolds with boundary, the graph Laplacian needs to be defined in order to gather sufficient information from the boundary, resulting in a slightly different form. First of all, we define a meta-metric $\widetilde{d}$ as follows and its properties are explained in Section $2$.

\begin{geodesic}\label{geodesic}
Let $d$ be the Riemannian distance function on $M$, and for $x,y\notin \partial M$, we define $\widetilde{d}$ as
\begin{eqnarray}\label{newd}
\widetilde{d}(x,y)=\inf_{z\in \partial M}\big(d(x,z)+d(z,y)\big),
\end{eqnarray}
and we call its minimizers by reflected geodesics. Notice that $\widetilde{d}$ satisfies the triangle inequality.
\end{geodesic}
We define the graph Laplacian $\Delta_{\Gamma}$ by
\begin{equation}\label{boundarylaplacian}
\Delta_{\Gamma}u(x_i) = \frac{2(n+2)}{\nu_n \rho^{n+2}}\bigg(\sum_{j: d(x_i,x_j)<\rho}\mu_j(u(x_j)-u(x_i))+\sum_{j: \widetilde{d}(x_i,x_j)<\rho}\mu_j(u(x_j)-u(x_i))\bigg),
\end{equation}
where $\nu_n$ is the volume of the unit ball in $\mathbb{R}^n$. Note that $\Delta_{\Gamma}$ is a finite dimensional linear operator$-$a matrix of dimension $N$, where $N$ is the total number of vertices. For $k\in\mathbb{N}$, denote the $k$-th eigenvalue of the graph Laplacian $-\Delta_{\Gamma}$ by $\lambda_k(\Gamma)$, and the $k$-th eigenvalue of the Laplace-Beltrami operator $-\Delta_M$ with the Neumann boundary condition by $\lambda_k(M)$. I prove that $\lambda_k(\Gamma)$ converges to $\lambda_k(M)$ for every nonnegative integer $k\leqslant N(X_{\varepsilon})-1$ as $\rho+\frac{\varepsilon}{\rho}\to 0$, where $N(X_{\varepsilon})$ denotes the number of points of the $\varepsilon$-net $X_{\varepsilon}$. More precisely,

\begin{main}\label{main}
Let $M\in \mathcal{M}_n(K_1,K_2,D,i_0)$ and $\Gamma_{\varepsilon,\rho}=\Gamma(X_{\varepsilon},\mu,\rho)$ be a weighted graph defined as above. Then for every nonnegative integer $k\leqslant N(X_{\varepsilon})-1$, there exists $\rho_0=\rho_0(n,i_0,D,K_1,K_2,\lambda_k(M))$ and $C=C(n,i_0,D,K_1,K_2)$, such that for any $\rho<\rho_0$, one has the following estimate:
$$|\lambda_k(\Gamma_{\varepsilon,\rho})-\lambda_k(M)|\leqslant C(\rho+\frac{\varepsilon}{\rho}+\rho\lambda_k(M)^{\frac{n}{2}+1})\lambda_k(M)+ C\rho .$$
Consequently, the eigenfunctions \textup{(see Remark \ref{remark})} of the graph Laplacians (\ref{boundarylaplacian}) converge to a respective eigenfunction of the Laplace-Beltrami operator with the Neumann boundary condition in $L^2(M)$.
\end{main}

\begin{remark}\label{remark}
We need a precise statement for the convergence of eigenfunctions of graph Laplacians, since they are actually discrete functions. Given a function $u$ on the vertices $X=\{x_i\}_{i=1}^N$, we can define a piecewise constant function $P^{\ast}u$ on $M$ by $P^{\ast}u=\sum_{i=1}^{N} u(x_i)1_{V_i}$. In Theorem \ref{main} and the two theorems that follow, the convergence of the eigenfuctions of graph Laplacians means exactly the convergence of the piecewise constant functions defined via the operator $P^{\ast}$.
\end{remark}

Now suppose $M$ is a more general metric-measure space which is isometrically glued out of $n$-dimensional compact Riemannian manifolds $M_l\in \mathcal{M}_n(K_1,K_2,D,i_0)$ without boundary or with smooth boundary, and every connected component of each nonempty gluing locus $M_j\cap M_l(j\neq l)$ is a $C^2$ submanifold of both $M_j$ and $M_l$ of codimension at least $1$ with piecewise $C^2$ boundary. Denote the whole gluing locus by $S=\cup_{j<l}(M_j\cap M_l)$. Without loss of generality, we assume each gluing locus is connected. Denote the Riemannian distance of $M_j$ by $d_j$, and $\widetilde{d}_j$ is defined with respect to $d_j$ as in Definition \ref{geodesic}. Denote the $r$-neighborhood of a nonempty gluing locus $M_j\cap M_l(j\neq l)$ within $M_j$ by $\Omega_r^{jl}=\{x\in M_j: d_j(x,M_j\cap M_l)<r\}$.
Then we define the \emph{reflected $r$-neighborhood} by
\begin{equation}\label{Omegajl}
\widehat{\Omega}_r^{jl}=\Omega_r^{jl}\sqcup \{ x\in M_j: d_j(x,M_j\cap M_l)<\widetilde{d}_j(x,M_j\cap M_l)<r\},
\end{equation}
where the union is a disjoint union. In (\ref{Omegajl}) we require $d_j$ to be strictly less than $\widetilde{d}_j$ for the sole reason that they both include the points which can be reached by geodesics touching the boundary, and such points are characterized by $d_j$ and $\widetilde{d}_j$ being equal. Note that we only defined $\widetilde{d}_j$ for interior points in (\ref{newd}). When either point is on the boundary, $\widetilde{d}_j$ by definition reduces to $d_j$. Hence if $M_j\cap M_l$ belongs to a manifold boundary, the reflected $r$-neighborhood (\ref{Omegajl}) simply reduces to the disjoint double copies of $\Omega_r^{jl}$. We impose the following assumption on the metric-measure spaces in question. The motivation behind this assumption is explained at the end of this section.

\begin{mainassumption}\label{A}
There exists $r_0>0$, such that for any $j\neq l$ and $M_j\cap M_l \neq \emptyset$, there exist homeomorphisms $\Phi_{jl}: \widehat{\Omega}_{r_0}^{jl}\to \widehat{\Omega}_{r_0}^{lj}$ satisfying: \\
$(1)$ $\Phi_{jl}|_{M_j\cap M_l}=Id|_{M_j\cap M_l}$; \\
$(2)$ For any vector $v \in T\widehat{\Omega}_{r_0}^{jl}$, the directional derivative $(\nabla \Phi_{jl})v$ of $\Phi_{jl}$ with respect to $v$ exists, and $||\nabla \Phi||_r\to 1$ as $r\to 0$,
where 
$$||\nabla\Phi||_{r}:=\max_{j\neq l} \sup_{v \in T\widehat{\Omega}_{r}^{jl}} \frac{|(\nabla\Phi_{jl})v|}{|v|}.$$ 
\end{mainassumption}

\begin{remark1}\label{remark1}
One can immediately see why we use the reflected neighborhood. Suppose the gluing locus is the boundary of one manifold part and does not intersect the boundary of the other manifold part. The standard neighborhood of the gluing locus within the first manifold is collar while tubular within the other manifold. The homeomorphism satisfying the Assumption clearly does not exist. It is also worth pointing out that the $r_0$-neighborhoods required in the Assumption do not have to be strict $r_0$-distance from the gluing locus. In fact, the existence of homeomorphisms satisfying the conditions above between any small open neighborhoods, regardless of collar or tubular, of the gluing locus will suffice.
\end{remark1}

Thanks to the dimension homogeneity, the Assumption is satisfied by a large class of metric-measure spaces in question. One obstacle for constructing such a homeomorphism is the Cauchy non-uniqueness for geodesics near a manifold boundary, for instance a boundary defined by $y=e^{-1/x^2}\sin(1/x)$. If geodesics near the gluing locus enjoy the Cauchy uniqueness property, or better do not intersect the boundary, the homeomorphisms satisfying the conditions are generated by geodesics from the gluing locus. For example, if a gluing locus has $C^2$ boundary and does not intersect with manifold boundaries, the homeomorphisms are straightforward to construct via geodesics within the normal coordinates of the gluing locus and its boundary. Due to the estimate on the length of the Jacobi field, we have $||(\nabla\Phi)v|-|v||\leqslant Cr^2 |v|$ for all vectors $v$ in tangent spaces over $r$-neighborhoods of the gluing locus for $r\ll 1$. This implies the Assumption with an explicit rate $o(1)=C(K_1)r^2$. If a gluing locus without boundary intersects a manifold boundary, the same construction can be done with the same rate. In this case, we simply choose the double copies of small neighborhoods within the one side of the gluing locus away from the manifold boundary (see Remark \ref{remark1}) to avoid the possible Cauchy non-uniqueness for geodesics near manifold boundaries. 
\indent In other words, the Assumption means that the $r$-neighborhood of the gluing locus within one manifold part can be mapped to $r$-neighborhoods within other manifold parts via almost isometries up to an infinitesimal error as $r \to 0$. This assumption gives us the desired symmetry: one can map a ball in tangent spaces via the homeomorphisms to almost another ball with little distortion. By virtue of the Assumption, we are able to recover the symmetry near the gluing locus as follows.
\begin{mirror}\label{mirror}
For any $\rho<\frac{1}{10}r_0^{\frac{4}{3}}\ll 1$ and a point $x\in M_j\subset M$ within $\rho^{\frac{3}{4}}$ distance (with respect to the Riemannian distance of $M_j$) from a nonempty gluing locus $M_j\cap M_l(j\neq l)$, we define the mirror image $x^l \in M_l$ of $x$, by $x^l=\Phi_{jl}(x)$. And denote $x^j=x$ for convenience. Notice that there could be two images via $\Phi_{jl}$ when the domain is a disjoint union. In this case, we choose any (one) image whose pre-image via $\Phi_{jl}$ belongs to the region defined by the standard Riemannian distance $d_l$ of $M_l$ instead of $\widetilde{d}_l$.
\end{mirror}
Note that the power $\frac{3}{4}$ in the definition is chosen to reconcile various inequalities to produce the convergence, which we will see later in the proof. Actually the power can be any number between $\frac{2}{3}$ and $1$ to guarantee the convergence. \\

\indent We construct the weighted graphs $\Gamma_{\varepsilon,\rho}=\Gamma(X_{\varepsilon},\mu,\rho)$ for a metric-measure space $M$ by choosing weighted graphs with the same parameters $\varepsilon\ll\rho$ for each manifold part $M_l$. Denote by $N(X_{\varepsilon})$ the number of points of the $\varepsilon$-net $X_{\varepsilon}$. The graph Laplacian in this case is defined as
\begin{equation}\label{metriclaplacian}
\Delta_{\Gamma}u(x_i) = \frac{2(n+2)}{\nu_n \rho^{n+2}}\sum_l \bigg(\sum_{j: d_l(x_i^l,x_j)<\rho}\mu_j(u(x_j)-u(x_i))+\sum_{j: \widetilde{d_l}(x_i^l,x_j)<\rho}\mu_j(u(x_j)-u(x_i))\bigg),
\end{equation}
where $\nu_n$ is the volume of the unit ball in $\mathbb{R}^n$ and $\widetilde{d}_l$ is defined as in Definition \ref{geodesic} with respect to the Riemannian distance $d_l$ of $M_l$. The $k$-th eigenvalue of the graph Laplacian $(1.4)$ is denoted by $-\lambda_k(\Gamma)$. On the other hand, the Laplacian eigenvalue problem with the Neumann boundary condition can be defined on the metric-measure space in question as shown in Section $3$, and denote the $k$-th eigenvalue by $\lambda_k(M)$ for $k\in\mathbb{N}$. For higher codimension there is no restriction at the gluing locus, implying that the spaces can be regarded as disjoint for the purpose of spectra (Lemma \ref{sobolev}(3)). Therefore we focus on the case of codimension $1$ gluing loci, where a natural Kirchhoff-type condition can be imposed: the sum of all normal derivatives at the gluing locus vanishes. I prove the following convergence result.

\begin{main2}\label{main2}
Suppose $M$ is a metric-measure space which is isometrically glued out of compact Riemannian manifolds of the same dimension without boundary or with smooth boundary. Assume every connected component of each nonempty gluing locus is a $C^2$ submanifold of codimension at least $1$ with piecewise $C^2$ boundary and satisfies the Assumption. Then $\lambda_k(\Gamma_{\varepsilon,\rho})$ converges to $\lambda_k(M)$ for every nonnegative integer $k\leqslant N(X_{\varepsilon})-1$ as $\rho+\frac{\varepsilon}{\rho}\to 0$. \\
Consequently the eigenfunctions of the graph Laplacians (\ref{metriclaplacian}) converge to a respective eigenfunction of the Laplacian eigenvalue problem with the Neumann boundary condition in $L^2(M)$.
\end{main2}

In particular, if each nonempty gluing locus is a codimension $1$ submanifold and does not intersect with manifold boundaries, or if a gluing locus without boundary intersects with manifold boundaries, the Assumption is satisfied as discussed earlier, which leads to an explicit convergence rate.

\begin{main3}\label{main3}
Suppose $M$ is a metric-measure space which is isometrically glued out of $m$ number of compact Riemannian manifolds $M_l\in \mathcal{M}_n(K_1,K_2,D,i_0)$. Assume every connected component of each nonempty gluing locus is a $C^2$ submanifold of codimension $1$ with $C^2$ boundary, and does not intersect with manifold boundaries unless the gluing locus is without boundary. Then for every nonnegative integer $k\leqslant N(X_{\varepsilon})-1$, there exists $\rho_0=\rho_0(n,m,i_0,D,K_1,K_2,r_0,\lambda_k(M))$ and $C=C(n,m,i_0,D,K_1,K_2,vol_{n-1}(S))$, such that for any $\rho<\rho_0$, one has the following estimate:
$$|\lambda_k(\Gamma_{\varepsilon,\rho})-\lambda_k(M)|\leqslant C(\rho^{\frac{1}{4}}+\frac{\varepsilon}{\rho}+\rho^{\frac{1}{4}}\lambda_k(M)^{\frac{n}{2}+1})\lambda_k(M) +C\rho^{\frac{1}{4}},$$
where $vol_{n-1}(S)$ is the $(n-1)$-dimensional volume of the gluing loci $S$, and $r_0$ explicitly depends only on $i_0$, the lower bound of the injectivity radius of each connected component of the gluing loci $S$ and the upper bound of the absolute value of sectional curvatures of $S$.
\end{main3}

From earlier discussions, Theorem \ref{main3} can be generalized to gluing loci with piecewise $C^2$ boundary and the case where the Cauchy uniqueness property holds for geodesics within a small neighborhood of the gluing locus. It is possible to consider even more general situations, but the convergence rate could heavily depend on the geometry of the gluing loci.

\subsection*{Ideas of the proof}

The blueprint of the proof follows from the original proof for closed manifolds in \cite{BIK}, which is constructing a discretization operator and a smoothing operator to compare functions on graphs with functions on manifolds via the min-max principle for eigenvalues. The main obstruction of proving the convergence in our cases is that the key estimates are optimal for closed manifolds and are heavily disrupted by any possible singularity (manifold boundary is viewed as a type of singularity), due to major changes in the symmetry and the volumes of small balls. The discretization operator is still there, but the same thing cannot be said about the smoothing operator, which needs to be constructed with the following considerations.\\
$(1)$ The smoothing operator should map every discrete function to a Lipschitz function. This ensures a proper connection via min-max formulae between discrete side and continuous side. For metric-measure spaces which are glued out of manifold parts, we use a suitable cut-off function to ensure continuity.\\
$(2)$ The smoothing operator should map constant discrete functions to constant smooth functions within a proper error. Suppose the averaging radius of the smoothing operator is $r\ll 1$ and it turns out that the error of the pointwise derivative has to be better than $r^{-1}$, which cannot be achieved without the help of some symmetry. For manifold boundaries, this can be resolved by adding back part of small balls defined via the meta-metric $\widetilde{d}$ and partially restoring the symmetry. The minimizers of $\widetilde{d}$ may touch or collide (not tangentially) with the boundary. We later show that there can only be one collision point, and such minimizers are characterized by the classical reflection, i.e. reversing the normal component and preserving the tangent component with respect to the tangent space of the boundary at the collision point. The situation of minimizers touching the boundary anywhere is difficult to analyze, for such minimizers are not governed solely by the geodesic equation anymore. However, we prove that this situation only generates higher order terms. Therefore, one can restrict attention to a reduced domain where only the simple behavior of the classical reflection occurs, and obtain key estimates similar to the ones for closed manifolds up to a higher order term. This is how we get the extra error term $\rho\lambda_k^{n/2+2}$ in Theorem \ref{main} and \ref{main3} compared with the result for closed manifolds. Since constant discrete functions on graphs are clearly eigenfunctions of the graph Laplacians with respect to zero eigenvalue, it is no surprise that eventually the spectra of graph Laplacians approximate the spectrum of the Laplace-Beltrami operator with the Neumann boundary condition.\\
\indent For metric-measure spaces glued out of manifolds, the singularities are much more serious. For instance, a standard metric ball near a gluing locus can intersect with multiple manifold parts and have an arbitrary shape and volume. This asymmetry of balls caused by the gluing significantly worsens estimates and consequently destroys the convergence. It is crucial to find a uniform and consistent way of averaging, which intuitively works as follows. Consider the simple space $M$ of three planar rectangles $M_1,M_2,M_3$ glued along an edge (e.g. book pages) and a point $x\in M_1$ near the edge. We take the two points corresponding to $x$ in the other two rectangles in the most natural way, since a small neighborhood of the gluing locus within $M_1$ is isometric to some small neighborhoods within $M_2,M_3$. Instead of averaging in a standard metric ball of $M$ around $x$, we take three seperate balls: the standard geodesic ball of $M_1$ around $x$, and the two standard geodesic balls of $M_2,M_3$ around the corresponding points of $x$ in the respective rectangle. In this way, the symmetry of balls is restored. Such construction can be done for the spaces in question as long as the Assumption is satisfied.\\

In Section $2$ we examine necessary facts regarding manifold boundaries. Theorem \ref{main} is proved in Section $3$. Section $4$ is devoted to studying the Laplacian eigenvalue problem on the metric-measure spaces in question. In Section $5$, we prove Theorem \ref{main2} and consequently Theorem \ref{main3}.

\begin{ack}
\textnormal{I am grateful to my advisor Dmitri Burago for introducing me to this problem and countless fruitful discussions and helpful guidance over the years. I would like to thank Mikhail Belkin for reviewing this work and valuable suggestions. The results in this paper were partially obtained during the author's Ph.D. at Penn State and were a part of the author's dissertation. }
\end{ack}

\section{Manifolds with boundary}

In this section, we suppose $M$ is an $n$-dimensional compact Riemannian manifold with smooth boundary $\partial M$, and $M\in \mathcal{M}_n(K_1,K_2,D,i_0)$ as in Definition \ref{bounded}. Let $X_{\varepsilon}=\{x_i\}_{i=1}^N$ for $x_i\notin \partial M$ be a finite $\varepsilon$-net on $M$ and $\Gamma_{\varepsilon,\rho}=\Gamma(X_{\varepsilon},\mu,\rho)$ with $\varepsilon\ll\rho\ll 1$ be the weighted graph defined in Section $1$. For simplicity, we write $X$ and $\Gamma$ for short. The space of functions on $X$ is equivalent to $\mathbb{R}^N$. For functions on $X$, we define
$$L^2(X)=\{u:X\to \mathbb{R}, ||u||_{L^2(X)}^2:=\sum_{i=1}^N \mu_i |u(x_i)|^2 <\infty\}.$$
In $L^2(X)$, we have an inner product $\langle u,v\rangle=\sum_{i=1}^N \mu_i u(x_i)v(x_i)$, for $u,v\in L^2(X)$. The graph Laplacian $-\Delta_{\Gamma}$ $(1.2)$ is self-adjoint and positive semi-definite with respect to the inner product, and its discrete energy is given by
\begin{equation}
||\delta u||^2=\langle -\Delta_{\Gamma}u,u\rangle= \frac{n+2}{\nu_n \rho^{n+2}}\sum_i \big( \sum_{j: d(x_i,x_j)<\rho}+\sum_{j: \widetilde{d}(x_i,x_j)<\rho}\big) \mu_i\mu_j |u(x_j)-u(x_i)|^2.
\end{equation}
Denote by $\lambda_k(\Gamma)$ the $k$-th eigenvalue of the graph Laplacian $-\Delta_{\Gamma}$ and $\lambda_0(\Gamma)=0$. The min-max principle applies:
$$\lambda_k(\Gamma)= \min_{L^{k+1}} \max_{u\in L^{k+1}-\{0\}} \frac{||\delta u||^2}{||u||^2},$$
where $L^{k+1}$ ranges over all $(k+1)$-dimensional subspaces of $L^2(X)$.\\
\indent On the other hand, $\Delta_{M}$ is the Laplace-Beltrami operator on $M$, and $\lambda_k(M)$ is the $k$-th eigenvalue of $-\Delta_M$ subject to the Neumann boundary condition with $\lambda_0(M)=0$. Similarly we have the following min-max formula,
$$\lambda_k(M)= \inf_{Q^{k+1}} \sup_{f\in Q^{k+1}-\{0\}} \frac{||\nabla f||_{L^2(M)}^2}{||f||_{L^2(M)}^2},$$
where $Q^{k+1}$ ranges over all $(k+1)$-dimensional subspaces of the Sobolev space $H^1(M)$.\\
\indent Now we study the properties of the meta-metric $\widetilde{d}$ in Definition \ref{geodesic}. By definition its minimizers are piecewise geodesics of $M$. The minimizers consist of two different types depending on the angle with respect to the tangent spaces of the boundary upon intersecting the boundary. If the angles for a minimizer at all intersection points are zero, then the minimizer is also a geodesic of $M$. We later prove in Lemma \ref{W} that this type of minimizers amounts to small measure when we restrict our attention to a sufficiently small normal neighborhood of the boundary. We are more interested in the other type of minimizers when they collide (not tangentially) with the boundary. Due to the first variation formula, such minimizer at the collision point is characterized by the \emph{classical reflection}, i.e. reversing the normal component and preserving the tangential component with respect to the tangent space of the boundary at the collision point. This is where the name \emph{reflected geodesic} comes from. However, this type of minimizers can still have multiple intersection points with the boundary. When this happens, the minimizer cannot collide with the boundary at another point, for it will fail to minimize the Riemannian distance. And we do not worry about the situations of minimizers touching the boundary, which only generates small measure. We choose our parameters to be smaller than the injectivity radius bound $i_0$ to avoid the situations where a reflected geodesic could intersect far parts of the boundary. And the minimizer between two sufficiently close points with respect to $\widetilde{d}$ is unique due to Corollary $3$ in \cite{ABB2}. Note that the result in \cite{ABB2} originally applies to geodesics of $M$, but the method is also valid for our reflected geodesics. Hence we proved the following lemma.
\begin{reflected}\label{reflected}
If $\widetilde{d}(x,y)<\frac{\pi}{\sqrt{K_1}}$ for $x,y\notin \partial M$, then $\widetilde{d}(x,y)$ is realized by a unique minimizing reflected geodesic. And if the reflected geodesic is not a geodesic of $M$, then it collides (not tangentially) with the boundary at exactly one point.
\end{reflected}


Our purpose is to extend the exponential map near the boundary to almost the whole tangent space. For a point $x\in M$ near the boundary, the geodesics from $x$ with the initial vector $v$ can touch or collide with the boundary. If a geodesic collides with the boundary, we reverse the normal component and preserve the tangential component to keep the geodesic extending. The Cauchy uniqueness for geodesics holds in this situation by Theorem $1$ in \cite{ABB1}. However, the geodesic can still intersect the boundary elsewhere, the effect of which needs to be measured. And we remove any vector whose image via the exponential map touches the boundary from our consideration. More precisely,

\begin{defW}\label{defW}
For $x \notin \partial M$ and $r\ll 1$, define $\mathcal{W}_r(x)$ to be the set of vectors $v\in T_x M$ of lengths at most $r$ such that the geodesic from $x$ with the initial vector $v$ satisfies either one of the following two conditions:\\
(1) the geodesic touches the boundary at the first intersection point with the boundary;\\
(2) the geodesic collides (not tangentially) with the boundary at the first intersection point, but intersects the boundary elsewhere after extending the geodesic via the classical reflection.
\end{defW}

And we prove that $\mathcal{W}_r(x)$ has small measure for $r\ll 1$.

\begin{W}\label{W}
$$m(\mathcal{W}_r (x)) \leqslant C(n,K_1) r^{n+1},\;\; \textrm{for } x\notin \partial M \textrm{ and }r\ll 1,$$
where $m(\cdot)$ denotes the $n$-dimensional Lebesgue measure in $\mathbb{R}^n$.
\end{W}
\begin{proof}
At the collision point, a geodesic of length at most $r$ can only intersect the boundary at another point if the angle is small than $K_1 r$ with respect to the tangent space of the boundary at the collision point. Therefore on the tangent space at the collision point, such vectors lie in simply the complement of two antipodal hyperspherical caps with the polar angle $\pi/2 -K_1 r$. Its volume is controlled by $C(n,K_1)r^{n+1}$. \\
\indent To translate the angle at the collision point to the angle at the initial point $x$, there are two types of angle changes to consider. The first type is the angle change brought by the exponential map. It is known that the norm of the differential of the exponential map is bounded by a factor $1\pm C(K_1)r^2$, so is the inner product, which by a straightforward calculation implies that the square of the angle is changed by $C(K_1)r^2$. We already know that the angle with respect to the tangent space at the collision point is bounded by $K_1 r$, and hence it follows that the angle at the point $x$ is bounded by $C(K_1)r$. The other type is the angle change brought by the change of reference tangent spaces of the boundary, which is bounded by $C(K_1)r$. Combining these, we know that the set of vectors in $\mathcal{W}_r(x)$ cannot exceed the angle $C(K_1)r$ with respect to any reference tangent space, and hence the volume estimate follows. Note that the considerations above already include the situation of geodesics touching the boundary (with zero angle).
\end{proof}

We define our extended exponential map with a small part $\mathcal{W}_r$ removed from the domain, and the behaviors of (reflected) geodesics near the boundary reduce to only the classical reflection. The exponential map on this domain is simple to analyze, for its images are standard geodesics in the interior with the only exception of the collision point. Otherwise without this reduced domain, geodesics touching the boundary bring significant trouble, because many nice properties of geodesics fail there such as the Cauchy uniqueness property and the smoothness of geodesics and Jacobi fields. Note that we may just remove the whole domain where the angle with respect to the tangent space at the first intersection point with the boundary is small than $K_1 r$, and sometimes this can be convenient for reasoning.\\

\indent Next we introduce notations for several domains we use throughout the paper.
\begin{ball}\label{ball}
Denote by $B_r(x)$ the standard open geodesic ball of radius $r$ centered at $x\in M$, and define $\widetilde{B}_r(x)=\{y: \widetilde{d}(x,y)<r\}$. We denote by $\{y:\widehat{d}(x,y)<r\}$ the disjoint union of $B_r(x)$ and $\widetilde{B}_r(x)$.\\
Define the reflected ball $\widehat{B}_r(x)$ to be the image of the reduced tangent space $\mathcal{B}_r(0)-\mathcal{W}_r(x)\subset T_x M$ via the extended exponential map, where $\mathcal{B}_r(0)$ is the ball of radius $r$ around the origin in the tangent space $T_x M$. Note that $\widehat{B}_r(x)$ is the disjoint union of two domains and is a subset of $\{y:\widehat{d}(x,y)<r\}$.
\end{ball}

We emphasize that the difference between $\widehat{B}_r(x)$ and $\{y:\widehat{d}(x,y)<r\}$. The former contains only the simple behavior of the classical reflection at the boundary with no geodesics touching the boundary, while the latter contains all possible behaviors. Note that $\widetilde{B}_r(x)\subset B_r(x)$, and $\widetilde{B}_r(x)=\emptyset$ if $x$ is far from the boundary. By Lemma \ref{reflected}, for $r<\min\{i_0,\pi/\sqrt{K_1}\}$ and any $x\notin \partial M$, $y\in \widehat{B}_r(x)\cap\widetilde{B}_r(x)$, there exists a unique reflected geodesic realizing $\widetilde{d}(x,y)$. Hence the extended exponential map $exp_x: \mathcal{B}_r(0)-\mathcal{W}_r(x) \to \widehat{B}_r(x)$ is a homeomorphism. From now on, we set the radii of all balls smaller than $\min\{i_0,\pi/\sqrt{K_1}\}$. We have the following estimate on the volume of the reflected ball $\widehat{B}_r(x)$.

\begin{Jacobi}\label{Jacobi}
Consider the reflected ball $\widehat{B}_r(x)$, where $d(x,\partial M)<r\ll 1$. Its volume has the following estimate
$$|vol(\widehat{B}_r(x))-\nu_{n} r^n|\leqslant C(n,K_1,K_2)r^{n+1},$$
where $\nu_n$ is the volume of the unit ball in $\mathbb{R}^n$. This is due to an estimate of the Jacobian $J_x(v)$ of the exponential map at $v\in \mathcal{B}_r(0)-\mathcal{W}_r(x)\subset T_x M$:
$$|J_x(v)-1| \leqslant C(n,K_1,K_2)|v|.$$
Furthermore, due to Lemma \ref{W} we have
$$|vol(\{y:\widehat{d}(x,y)<r\})-\nu_{n} r^n|\leqslant C(n,K_1,K_2)r^{n+1}.$$
\end{Jacobi}
The volume estimate is due to an estimate on the Jacobi field. Since we are working in $\widehat{B}_r(x)$, the geodesics behave just like in closed manifolds with the only exception at the collision point. A straightforward calculation shows that the length of the Jacobi field is of $C^2$ at the collision point. Unlike for closed manifolds, the third derivative of the length of the Jacobi field does not vanish and is bounded by $C(K_1,K_2)$, which generates the first error term. As a comparison, if $x$ is far from the boundary, the first error term would be $C(n, K_1)r^{n+2}$.

We are now in place to prove Theorem \ref{main}. We use the reflected ball $\widehat{B}_r(x)$ instead of the standard geodesic ball so that the geodesics can be extended upon colliding with the boundary. And the discrete energy $(2.1)$ matches the energy estimates in terms of the reflected balls up to a higher order term. Almost the same proof in \cite{BIK} works because of the following two key facts: \\
(1) The volumes of reflected balls are preserved along the geodesics near the boundary up to a higher order term. (Lemma \ref{Jacobi}) \\
(2) $\mathcal{W}_r(x)$ has small measure. (Lemma \ref{W})

\section{Proof of Theorem \ref{main}}

In this section, we prove Theorem \ref{main} by obtaining the upper and lower bound for the eigenvalues of the graph Laplacian (\ref{boundarylaplacian}) in Lemma \ref{upperbound} and \ref{lowerbound}. 

\subsection*{Upper bound for $\lambda_k(\Gamma)$}

Given a weighted graph $\Gamma=\Gamma(X,\mu,\rho)$. Define the discretization operator $P:L^2(M)\to L^2(X)$ by
$$Pf(x_i)=\mu_i^{-1}\int_{V_i} f(x)dx,$$
and $P^{\ast}: L^2(X)\to L^2(M)$ by
$$P^{\ast}u=\sum_{i=1}^{N}u(x_i)1_{V_i}.$$
It immediately follows that $||P^{\ast}u||_{L^2 (M)}=||u||_{L^2 (X)}$. For $f\in L^2(M)$, define
\begin{equation}\label{Erformanifolds}
\widehat{E}_r(f,V)=\int_V \int_{\widehat{B}_r (x)} |f(y)-f(x)|^2 dy dx,
\end{equation}
and $\widehat{E}_r(f)=\widehat{E}_r(f,M)$, where $\widehat{B}_r(x)$ is defined in Definition \ref{ball}. The following two lemmas enable us to obtain the upper bounds of the discrete norm and energy in terms of their continuous counterparts.

\begin{2.2}\label{2.2}
For $r<2\rho$ and $f\in C^{\infty}(M)$, one has
$$\widehat{E}_r(f)\leqslant (1+Cr)\frac{\nu_n}{n+2} r^{n+2} ||\nabla f||^2_{L^2(M)}.$$
\end{2.2}

\begin{proof}
Take two exact copies of $M$ and glue them along the boundary via the identity map. Denote the space by $\widetilde{M}$. Consider a function $\widetilde{f}$ as two exact copies of $f$. The distance between a point $x$ and $y\in \widehat{B}_r(x)$ on the other copy are achieved by reflected geodesics characterized by the classical reflection, as discussed in Section $2$. $\widetilde{M}$ can be considered as a manifold without boundary. By virtue of Lemma \ref{Jacobi}, the volumes of reflected balls on $\widetilde{M}$ with small radius are preserved along reflected geodesics. Therefore Lemma $3.3$ in \cite{BIK} applies:
$$\int_{\widetilde{M}}\int_{\widehat{B}_r (x)}|\widetilde{f}(y)-\widetilde{f}(x)|^2 dy dx \leqslant (1+Cr) \frac{\nu_n}{n+2} r^{n+2} ||\nabla \widetilde{f}||^2_{L^2(\widetilde{M})},$$
where the factor $(1+Cr)$ comes from the Jacobian estimate in Lemma \ref{Jacobi}. Take half on both sides of the inequality and our lemma follows.
\end{proof}

\begin{2.3}\label{2.3}
For $r\in(\varepsilon, 2\rho)$ and $f\in C^{\infty}(M)$, one has
$$\int_{V_i} |f(x)-P f(x_i)|^2 dx \leqslant \frac{C}{\nu_n (r-\varepsilon)^n} \int_{V_i}\int_{\{y:\widehat{d}(x,y)<r\}}|f(y)-f(x)|^2 dy dx,$$
where $\{y:\widehat{d}(x,y)<r\}$ is defined in Definition \ref{ball}.
\end{2.3}

\begin{proof}
Fix $x,y\in V_i$ and $r>\varepsilon$, and consider $U=\{z:\widehat{d}(x,z)<r\}\cap\{z:\widehat{d}(y,z)<r\}$. Observe that $U$ contains a ball (defined via $\widehat{d}$) of radius $r-\varepsilon$ centered at the midpoint (with respect to $d$) between $x$ and $y$. Thus the volume of $U$ has a lower bound $vol(U)\geqslant C\nu_n(r-\varepsilon)^n$ by Lemma \ref{Jacobi}. The rest of the proof is exactly the same as Lemma $3.4$ in \cite{BIK}.
\end{proof}

Now with these two lemmas, we are able to bound the discrete norm and energy by their continuous counterparts. Set $r=(n+1)\varepsilon$, and by Lemma \ref{2.2} and Lemma \ref{2.3} we have
\begin{eqnarray}\label{value1}
||f-P^{\ast}P f||_{L^2}^2 &=& \sum_i \int_{V_i} |f(x)-P f(x_i)|^2 dx \nonumber \\
&\leqslant & \frac{C}{\nu_n (r-\varepsilon)^n} \int_{M}\int_{\{y:\widehat{d}(x,y)<r\}}|f(y)-f(x)|^2 dy dx \nonumber \\
&\leqslant&  \frac{C}{\nu_n (r-\varepsilon)^n} \bigg(\widehat{E}_{r}(f,M) + 2\int_{M}\int_{exp(\mathcal{W}_r(x))} |f(y)-f(x)|^2 dy dx \bigg) \nonumber \\
&\leqslant & C(\frac{r}{r-\varepsilon})^n r^2 ||\nabla f||^2_{L^2(M)} + \frac{C}{(r-\varepsilon)^n}r^2 m(\mathcal{W}_r(x)) ||\nabla f||_{L^{\infty}(M)}^2\nonumber \\
&\leqslant & C\varepsilon^2 ||\nabla f||^2_{L^2(M)}+C\varepsilon^3 ||\nabla f||^2_{L^{\infty}(M)}.
\end{eqnarray}
Keep in mind the difference between $\widehat{B}_r(x)$ and $\{y: \widehat{d}(x,y)<r\}$. Notice that we actually need to add $\mathcal{W}_r(x)$ back twice, for both $d$ and $\widetilde{d}$ include points which are reached by geodesics touching the boundary. On the other hand, by the definition of $P$ and the Cauchy-Schwarz inequality,
\begin{eqnarray*}\label{Edelta1}
||\delta(P f)||^2 &=& \frac{n+2}{\nu_n \rho^{n+2}}\sum_i ( \sum_{j: d(x_i,x_j)<\rho}+\sum_{j: \widetilde{d}(x_i,x_j)<\rho}) \mu_i\mu_j |P f(x_j)-P f(x_i)|^2  \nonumber \\
&\leqslant& \frac{n+2}{\nu_n \rho^{n+2}}\sum_i ( \sum_{j: d(x_i,x_j)<\rho}+\sum_{j: \widetilde{d}(x_i,x_j)<\rho})\int_{V_i}\int_{V_j} |f(x)-f(y)|^2 dy dx \nonumber \\
&=& \frac{n+2}{\nu_n \rho^{n+2}} \int_{M} \int_{\bigcup_{\widehat{d}(x_i,x_j)<\rho}V_j, x\in V_i}|f(y)-f(x)|^2 dy dx.
\end{eqnarray*}
Also here notice that the sum defined by $\widehat{d}$ involves situations where geodesics touch the boundary or geodesics intersect the boundary elsewhere with the presence of one collision point, which are removed from the definition of $\widehat{B}_r(x)$. To control the right-hand side in terms of the reflected ball $\widehat{B}_r(x)$, we need to add $\mathcal{W}_r(x)$ back (twice). Thanks to the volume estimate in Lemma \ref{W}, this only generates a higher order term. Since $\widetilde{d}$ satisfies the triangle inequality, we know 
\begin{equation}\label{triangle0}
\{y: \widehat{d}(x_i,y)<\rho-4\varepsilon\} \subset \bigcup_{j: \widehat{d}(x_i,x_j)<\rho} V_j \subset \{y: \widehat{d}(x_i,y)<\rho+4\varepsilon\},
\end{equation}
and by Lemma \ref{2.2} we get
\begin{eqnarray}\label{Edelta3}
||\delta(P f)||^2 &\leqslant& \frac{n+2}{\nu_n \rho^{n+2}} \int_M \int_{\{y:\widehat{d}(x,y)<\rho+4\varepsilon\}} |f(y)-f(x)|^2dydx \nonumber \\
&\leqslant&  \frac{n+2}{\nu_n \rho^{n+2}} \widehat{E}_{\rho+4\varepsilon}(f)+ \frac{2(n+2)}{\nu_n \rho^{n+2}} \int_M \int_{exp_x(\mathcal{W}_{2\rho}(x))} |f(y)-f(x)|^2dydx \nonumber \\
&\leqslant& (1+C\rho+C\frac{\varepsilon}{\rho})||\nabla f||^2_{L^2} + Cvol(M)\rho||\nabla f||^2_{L^{\infty}}.
\end{eqnarray}

We obtain the upper bound for $\lambda_k(\Gamma)$ in the following lemma.

\begin{upperbound}\label{upperbound}
For sufficiently small $\rho$, we have
$$\lambda_k(\Gamma)\leqslant (1+C\rho+C\frac{\varepsilon}{\rho}+C\rho\lambda_k^{\frac{n}{2}+1})\lambda_k(M) + C\rho,$$
where the constants explicitly depend on $n,i_0,D,K_1,K_2$.
\end{upperbound}

\begin{proof}
We choose the first $k+1$ eigenfunctions $f_0,\cdots,f_{k}$ of $-\Delta_M$ with respect to the eigenvalues $\lambda_0,\lambda_1,\cdots,\lambda_{k}$. We consider the linear space spanned by $Pf_0,\cdots,Pf_{k}$ which is a subspace of $L^2(X)$, and we first prove this subspace is $(k+1)$-dimensional for sufficiently small $\varepsilon$.\\
\indent The eigenfunction $f_k$ is smooth and satisfies the eigenvalue equation pointwise. Then by G$\mathring{\textrm{a}}$rding's inequality, we get an upper bound in any $H^{2s}$ norm on $M$:
\begin{eqnarray*}
||f_k||_{H^{2s}(M)}&\leqslant& C(n,s)(||(-\Delta)^s f_k||_{L^2(M)}+||f_k||_{L^2(M)})\\
&=& C(n,s)(\lambda_k^s+1)||f_k||_{L^2(M)}.
\end{eqnarray*}
By the Sobolev embedding theorem (and the partition of unity), we have $H^{2s}(M)\subset C^{1}(M)$ for $2s>n/2+1$. Choose $2s=n/2+2$ and hence we get
$$||f_k||_{C^1(M)}\leqslant C(n,i_0,vol(M))||f_k||_{H^{n/2+2}(M)} \leqslant C(n,i_0,vol(M)) (\lambda_k^{n/4+1}+1)||f_k||_{L^2(M)}.$$
The volume of $M$ is bounded by a constant depending on $n,D,K_1$. Therefore, we obtain
\begin{equation}\label{morrey0}
||\nabla f_k||_{L^{\infty}(M)}^2 \leqslant C(n,i_0,D,K_1)(\lambda_k^{n/2+2}+1)||f_k||^2_{L^2(M)}.
\end{equation}
Then by (\ref{value1}) and (\ref{morrey0}), for any $f$ in the first $k+1$ eigenspaces we get
\begin{equation}\label{injective0}
||Pf||_{L^2(X)} \geqslant \big(1-C\varepsilon\sqrt{\lambda_k^{n/2+2}+1}\,\big)||f||_{L^2(M)}.
\end{equation}
Hence for sufficiently small $\varepsilon$ depending on $n,i_0,D,K_1,\lambda_k$, the linear operator $P$ is injective, which implies that the linear subspace spanned by $Pf_0,\cdots,Pf_{k}$ is $(k+1)$-dimensional.\\
\indent For any $f\in span\{f_0,\cdots,f_{k}\}$, by (\ref{Edelta3}) and (\ref{morrey0}) we get
\begin{eqnarray*}
||\delta(Pf)||^2 &\leqslant&(1+C\rho+C\frac{\varepsilon}{\rho})||\nabla f||^2_{L^2} + Cvol(M)\rho||\nabla f||^2_{L^{\infty}} \\
&\leqslant& (1+C\rho+C\frac{\varepsilon}{\rho})||\nabla f||^2_{L^2(M)}+C\rho(\lambda_k^{n/2+2}+1)||f||^2_{L^2(M)}.
\end{eqnarray*}
Combine this inequality with (\ref{injective0}), and for sufficiently small $\varepsilon, \rho$ we obtain
$$\frac{||\delta(Pf)||^2}{||Pf||^2_{L^2}} \leqslant (1+C\rho+C\frac{\varepsilon}{\rho})\lambda_k + C\rho(\lambda_k^{n/2+2}+1),$$
which implies the upper bound for $\lambda_k(\Gamma)$ due to the min-max principle. 
\end{proof}

\subsection*{Lower bound for $\lambda_k(\Gamma)$}

Next we deal with the lower bound. We fix $r\ll 1$ and consider a kernels $k_r: M\times M\to \mathbb{R}_+$ defined by
$$k_r(x,y)=r^{-n}\phi(\frac{\widehat{d}(x,y)}{r})\chi_{\widehat{B}_r(x)}.$$
The function $\phi:\mathbb{R}_{\geq 0}\to \mathbb{R}_{\geq 0}$ is defined as $\phi(t)=\frac{n+2}{2\nu_n}(1-t^2)$ if $t\in[0,1]$ otherwise $0$. Note that the normalization constant $\frac{n+2}{2\nu_n}$ is chosen so that $\int_{\mathbb{R}^n}\phi(|x|)dx=1$. We define the associated integral operator $\Lambda_r: L^2(M)\to C^{0,1}(M)$ given by
$$\Lambda_r f(x)=\int_{M}k_r(x,y)f(y)dy.$$
Note that $k_r(x,y)$ is not continuous, but its integral with respect to $y$ is continuous in $x$, since the reflected ball $\widehat{B}_r(x)$ varies continuously. From earlier discussions in Section $2$, we know that for any $r<\min\{i(M),\pi/\sqrt{K_1}\}$, the extended exponential map $exp_x: \mathcal{B}_r(0)-\mathcal{W}_r(x) \to \widehat{B}_r(x)$ is a homeomorphism. A direct computation yields
\begin{equation}\label{diffk}
\nabla_x k_r(x,y)=\frac{n+2}{\nu_n r^{n+2}}exp_x^{-1}(y), \textrm{ for }y\in \widehat{B}_r(x).
\end{equation}
\indent First we find out how much the image of a constant discrete function via the operator $\Lambda_r$ approximates the constant function on manifolds.

\begin{thetax}\label{thetax}
Define $\theta(x)=\Lambda_r(1_M)$. For almost every $x\in M$, one has
$$|\theta(x)-1|\leqslant Cr,$$
and
$$|\nabla\theta(x)|\leqslant C.$$
\end{thetax}

\begin{proof}
By the definition of the operator $\Lambda_r$ and the kernel $k_r$, we have
\begin{eqnarray*}
\theta(x) &=& \int_{\widehat{B}_r(x)}k_r(x,y)dy = r^{-n} \int_{\widehat{B}_r(x)}\phi(\frac{\widehat{d}(x,y)}{r})dy \\
&=& r^{-n}\int_{\mathcal{B}_r(0)-\mathcal{W}_r(x)\subset T_x M} \phi(r^{-1}|v|)J_x(v) dv,
\end{eqnarray*}
where $\mathcal{W}_r(x)\subset T_x M$ is defined in Definition \ref{defW}. Since $\int_{\mathcal{B}_r(0)}\phi(r^{-1}|v|) dv=r^n$ by our choice of $\phi$, the first estimate follows from Lemma \ref{W} and the Jacobian estimate in Lemma \ref{Jacobi}. As for the second inquality, we differentiate $\theta(x)$ using (\ref{diffk}):
\begin{eqnarray*}
\nabla \theta(x) &=& \frac{n+2}{\nu_n r^{n+2}}\int_{\widehat{B}_r(x)}exp_x^{-1}(y) dy \\
&=& \frac{n+2}{\nu_n r^{n+2}}\int_{\mathcal{B}_r(0)} v J_x(v) dv-\frac{n+2}{\nu_n r^{n+2}}\int_{\mathcal{W}_r(x)} v J_x(v) dv.
\end{eqnarray*}
The second term is bounded by a constant $C$ by virtue of Lemma \ref{W}. For the first term, since $\int_{\mathcal{B}_r(0)}v dv=0$ due to the symmetry, we can replace $J_x(v)$ by $J_x(v)-1$. Then the Jacobian estimate (Lemma $\ref{Jacobi}$) yields the estimate.
$$\frac{n+2}{\nu_n r^{n+2}}|\int_{\mathcal{B}_r(0)} v J_x(v) dv| \leqslant \frac{n+2}{\nu_n r^{n+2}}\int_{\mathcal{B}_r(0)} |v| Cr dv \leqslant C.$$
\end{proof}
Note that the extra power of $r$ generated by the symmetry of balls in tangent spaces is crucial; the proof will not work without it. Now we define $I_r(f)=\theta^{-1}\Lambda_r(f)$ for a function $f\in L^2(M)$ which we will replace at the end by the piecewise constant function $P^{\ast}u$ for a discrete function $u\in L^2(X)$. The following two lemmas estimate the norm and energy of $I_r(f)$ in terms of $\widehat{E}_r(f)$ defined in (\ref{Erformanifolds}).

\begin{1}\label{1}
For $f\in L^2 (M)$, one has
$$||I_r f-f||_{L^2(M)}^2\leqslant \frac{C}{\nu_n r^n}\widehat{E}_r(f).$$
\end{1}

\begin{proof}
The proof is straightforward. One can refer to Lemma $5.4$ in \cite{BIK} or Lemma \ref{new2}(1).
\end{proof}

\begin{2}\label{2}
For $f\in L^2 (M)$, one has
$$||\nabla(I_r f)||_{L^2(M)}^2\leqslant (1+Cr)\frac{n+2}{\nu_n r^{n+2}}\widehat{E}_r(f).$$
\end{2}

\begin{proof}
For any fixed $x_0\notin \partial M$, by the definition of $\theta$ we have
\begin{equation}\label{A6}
\theta^{-1}\Lambda_r f(x)=f(x_0)+\theta^{-1}\int_{M} (f(y)-f(x_0))k_r(x,y)dy.
\end{equation}
Differentiating (\ref{A6}) and evaluating it at the point $x_0$ yields
\begin{eqnarray*}
\nabla(\theta^{-1}\Lambda_r f)(x_0) &=& \theta^{-1}(x_0)\int_{M} (f(y)-f(x_0))\frac{\partial}{\partial x}k_r(x_0,y)dy \\
&+& \nabla(\theta^{-1})(x_0)\int_{M} (f(y)-f(x_0))k_r(x_0,y)dy \\
&=& \theta^{-1}(x_0)A_1(x_0)+A_2(x_0),
\end{eqnarray*}
where
$$A_1(x)=\int_{\widehat{B}_r(x)} (f(y)-f(x))\frac{\partial}{\partial x}k_r(x,y)dy,$$
and
$$A_2(x)=\nabla(\theta^{-1})\int_{\widehat{B}_r(x)} (f(y)-f(x))k_r(x,y)dy.$$
Since $|\theta-1|\leqslant Cr$, we have
\begin{equation}\label{A1A2}
||\nabla(\theta^{-1}\Lambda_r f)||_{L^2}\leqslant (1+Cr)||A_1||_{L^2}+||A_2||_{L^2}.
\end{equation}
First, we estimate $||A_1||_{L^2}$. For $x\notin \partial M$, set $\omega=\frac{A_1(x)}{|A_1(x)|}$ and we have
\begin{eqnarray*}
|A_1(x)|&=& \langle A_1(x),w\rangle \\
&=&\frac{n+2}{\nu_n r^{n+2}}\int_{\widehat{B}_r(x)}(f(y)-f(x))\langle exp_x^{-1}(y),w\rangle dy \\
&=& \frac{n+2}{\nu_n r^{n+2}}\int_{\mathcal{B}_r(0)-\mathcal{W}_r(x)\subset T_x M}(f(exp_x(v))-f(x)) \langle v,w\rangle  J_x(v) dv.
\end{eqnarray*}
By the Cauchy-Schwarz inequality,
\begin{eqnarray*}
|A_1(x)|^2 &\leqslant& (\frac{n+2}{\nu_n r^{n+2}})^2(\int_{\mathcal{B}_r(0)-\mathcal{W}_r(x)}|f(exp_x(v))-f(x)|^2 |J_x(v)|^2 dv)(\int_{\mathcal{B}_r(0)} \langle v,w\rangle^2 dv) \\
&=& \frac{n+2}{\nu_n r^{n+2}}\int_{\mathcal{B}_r(0)-\mathcal{W}_r(x)}|f(exp_x(v))-f(x)|^2 |J_x(v)|^2 dv.
\end{eqnarray*}
The last equality is due to the following argument. Expand $w=w_1$ to an orthonormal basis $w_j$ of $T_x M$ for $j=1,\cdots,n$. Since $|v|^2=\sum_j \langle v,w_j\rangle^2$, we have $\int_{|v|<r}|v|^2 dv=\sum_j \int_{|v|<r}\langle v,w_j\rangle^2 dv$. Due to the symmetry of $\mathcal{B}_r(0)\subset T_x M$, we have for any $j$,
\begin{equation}\label{sublemma}
\int_{|v|<r}\langle v,w_j\rangle^2 dv=\frac{1}{n}\int_{|v|<r}|v|^2 dv=\frac{\nu_n}{n+2}r^{n+2}.
\end{equation}
Hence by the Jacobian estimate (Lemma \ref{Jacobi}),
\begin{eqnarray*}
||A_1(x)||_{L^2}^2 &\leqslant & \frac{n+2}{\nu_n r^{n+2}}\int_M \int_{\widehat{B}_r(x)}|f(y)-f(x)|^2 |J_x(exp_x^{-1}(y))|^2 dy dx\\
&\leqslant& (1+Cr) \frac{n+2}{\nu_n r^{n+2}}\widehat{E}_r(f).
\end{eqnarray*}
Next we estimate $||A_2||_{L^2}$. Due to $|\theta|\leqslant C$, $|\nabla\theta|\leqslant C$ (Lemma \ref{thetax}) and $k_r\leqslant\frac{C}{\nu_n r^n}$, we get
\begin{eqnarray*}
|A_2(x)|^2 &\leqslant& |\nabla(\theta^{-1})|^2 (\int_{\widehat{B}_r(x)}k_rdy)(\int_{\widehat{B}_r(x)}|f(y)-f(x)|^2k_rdy)\\
&\leqslant & |\nabla(\theta^{-1})|^2 \theta(x) \frac{C}{\nu_n r^n} \int_{\widehat{B}_r(x)}|f(y)-f(x)|^2 dy \\
&\leqslant & \frac{C}{\nu_n r^n} \int_{\widehat{B}_r(x)}|f(y)-f(x)|^2 dy.
\end{eqnarray*}
Therefore,
$$||A_2(x)||_{L^2}^2 \leqslant \frac{C}{\nu_n r^n} \int_M\int_{\widehat{B}_r(x)}|f(y)-f(x)|^2 dy dx = \frac{C}{\nu_n r^n} \widehat{E}_r(f). $$
Finally combining the estimates on the norm of $A_1$ and $A_2$, by (\ref{A1A2}) we obtain
\begin{eqnarray*}
||\nabla(\theta^{-1}\Lambda_r f)||_{L^2} &\leqslant& (1+Cr)||A_1||_{L^2}+||A_2||_{L^2} \\
&\leqslant& ((1+Cr)^{\frac{3}{2}}+Cr)\sqrt{\frac{n+2}{\nu_n r^{n+2}}\widehat{E}_r(f)} \\
&\leqslant& (1+Cr)\sqrt{\frac{n+2}{\nu_n r^{n+2}}\widehat{E}_r(f)}.
\end{eqnarray*}
\end{proof}
Now we replace the $L^2(M)$ function $f$ in the previous two lemmas by the piecewise constant function $P^{\ast}u$ for a discrete function $u\in L^2(X)$. The only thing left is to estimate $\widehat{E}_{r}(P^{\ast}u)$ in terms of the discrete energy $||\delta u||^2$. By the definition of discrete energy $||\delta u||^2$, we have
\begin{eqnarray*}
||\delta u||^2&=& \frac{n+2}{\nu_n \rho^{n+2}}\sum_i ( \sum_{j: d(x_i,x_j)<\rho}+\sum_{j: \widetilde{d}(x_i,x_j)<\rho}) \mu_i\mu_j |u(x_j)-u(x_i)|^2 \\
              &=& \frac{n+2}{\nu_n \rho^{n+2}} \int_{M} \int_{\bigcup_{\widehat{d}(x_i,x_j)<\rho}V_j, x\in V_i}|P^{\ast}u(y)-P^{\ast}u(x)|^2 dy dx \\
              &\geqslant& \frac{n+2}{\nu_n \rho^{n+2}} \widehat{E}_{\rho-4\varepsilon}(P^{\ast}u),
\end{eqnarray*}
where the last inequality is due to the triangle inequality (\ref{triangle0}). This part is simpler than the part we did in (\ref{Edelta3}), since the discrete energy already contains more information than $\widehat{E}_r$. Therefore we get
\begin{equation}\label{Edelta2}
\widehat{E}_{\rho-4\varepsilon}(P^{\ast}u)\leqslant \frac{\nu_n \rho^{n+2}}{n+2}||\delta u||^2.
\end{equation}

Now set $r=\rho-4\varepsilon$, and define $Iu=I_{\rho-4\varepsilon}(P^{\ast}u)=\theta^{-1}\Lambda_{\rho-4\varepsilon}(P^{\ast}u)$, for $u\in L^2(X)$. Insert (\ref{Edelta2}) into the estimates in Lemma \ref{1} and \ref{2}, and we can finally bound the norm and energy of $Iu$ by the discrete norm and energy of $u$.
\begin{eqnarray}\label{eigenfunction0}
||I u-P^{\ast}u||_{L^2(M)}^2 & \leqslant & \frac{C}{\nu_n (\rho-4\varepsilon)^n}\widehat{E}_{\rho-4\varepsilon}(P^{\ast}u) \nonumber \\
& \leqslant & C (\frac{\rho}{\rho-4\varepsilon})^n \rho^2||\delta u||^2 \nonumber \\
 &\leqslant & C\rho^2 ||\delta u||^2,
\end{eqnarray}
and
\begin{eqnarray}
||\nabla(I u)||_{L^2(M)}^2&\leqslant& (1+C\rho)\frac{n+2}{\nu_n (\rho-4\varepsilon)^{n+2}}\widehat{E}_{\rho-4\varepsilon}(P^{\ast}u) \nonumber \\
              &\leqslant& (1+C\rho)(\frac{\rho}{\rho-4\varepsilon})^{n+2}||\delta u||^2 \nonumber \\
              &\leqslant& (1+C\rho+C\frac{\varepsilon}{\rho})||\delta u||^2.
\end{eqnarray}

Hence we obtain the lower bound for $\lambda_k(\Gamma)$ in the same way as Lemma \ref{upperbound}.

\begin{lowerbound}\label{lowerbound}
For sufficiently small $\rho$, we have
$$\lambda_k(\Gamma)\geqslant (1-C\rho-C\frac{\varepsilon}{\rho}-C\rho\sqrt{\lambda_k})\lambda_k(M),$$
where the constants explicitly depend on $n,K_1,K_2$.
\end{lowerbound}

\begin{proof}[Proof of Theorem \ref{main}]
The convergence of eigenvalues is given by Lemma \ref{upperbound} and \ref{lowerbound}. The convergence of eigenfunctions with an explicit rate is a direct consequence of (\ref{eigenfunction0}) and the convergence of eigenvalues. The proof is straightforward and we only provide a sketch here. One may refer to the proof of Theorem $4$ in \cite{BIK}.

Suppose $u_0,u_1,\cdots,u_k$ are the first $k+1$ orthonormal eigenvectors of the graph Laplacian $\Delta_{\Gamma}$, and consider the linear subspace spanned by $\{Iu_0,\cdots,Iu_k\}$. For sufficiently small $\varepsilon,\rho$, the operator $I$ is injective and hence the linear subspace is also $(k+1)$-dimensional. Furthermore by (\ref{eigenfunction0}), the operator $I$ on the subspace almost preserves $L^2$ norms up to an error of $C\rho\sqrt{\lambda_k(\Gamma)}$ and hence almost preserves inner products. This implies that $I{u_k}$ is almost orthogonal to $Iu_0,\cdots,Iu_{k-1}$. Since $u_0$ is a constant discrete function, $Iu_0$ is a constant function by definition and is hence the first eigenfunction. For $Iu_1$, we know that $||\nabla(Iu_1)||_{L^2}^2/||Iu_1||_{L^2}^2$ is bounded from above almost by $\lambda_1(\Gamma)$ for sufficiently small $\varepsilon,\rho$ and hence bounded almost by $\lambda_1(M)$ due to the convergence of eigenvalues. Since the projection of $Iu_1$ onto $Iu_0$ i.e. the first eigenfunction is small, it follows that the projection of $Iu_1$ onto any eigenspace with respect to an eigenvalue larger than $\lambda_1(M)$ is also small. This implies that $Iu_1$ is close to an eigenfunction with respect to the eigenvalue $\lambda_1(M)$. Repeating this process and one can prove that for any $k$, $Iu_k$ is close in $L^2$ to an eigenfunction with respect to the eigenvalue $\lambda_k(M)$. The convergence of eigenfunctions follows from the fact that $P^{\ast}u_k$ and $Iu_k$ is close in $L^2$ by (\ref{eigenfunction0}).

\end{proof}

\section{Metric-measure spaces glued out of manifolds}

In this section, we suppose $M$ is a metric-measure space which is isometrically glued out of compact Riemannian manifolds of the same dimension. More precisely, consider a metric-measure space $M$ satisfying the following conditions:\\
$(1)$ $M=\cup_{l=1}^{m} M_l$, where each $M_l$ is an $n$-dimensional compact Riemannian manifold without boundary or with smooth boundary, and $M_l \in \mathcal{M}_n(K_1,K_2,D,i_0)$ as in Definition \ref{bounded};\\
$(2)$ For any $j<l$, every connected component of each nonempty gluing locus $M_{j}\cap M_l$ is a $C^2$ submanifold of both $M_j$ and $M_l$ of codimension at least $1$ with piecewise $C^2$ boundary. Denote the whole gluing locus by $S=\cup_{j<l}(M_{j}\cap M_l)$ and $\partial M=\cup_l \partial M_l-S$. And the Riemannian metrics of $M_j$ and $M_l$ agree at tangent spaces $T_x (M_j\cap M_l)$ for any $x\in M_j\cap M_l\neq\emptyset$. \\
\indent We start by discussing the space of functions on $M$. If $u$ is a function on $M$, denote by $u_l$ the restriction of $u$ onto $M_l$.

\begin{function}
$C^{\infty}(M)=\{u:  u_l\in C^{\infty}(M_l-S)\cap C^0(M),\textrm{ for any }l\,\}$;\\
$\widehat{W}^{k,p}(M)=\{u:  u_l \in W^{k,p}(M_l)\}$;\\
$H^1(M)=W^{1,2}(M)=\{\overline{C^{\infty}(M)}\subset \widehat{W}^{1,2}(M)\}$.\\
The derivatives of a function at a point outside the gluing loci are taken with respect to the Riemannian structure where the point lies. The $L^p$ and $W^{k,p}$ norms are the sum of norms with respect to the canonical Riemannian volume form $dv_l$ of each $M_l$.
\end{function}

\begin{sobolev}\label{sobolev}
$(1)$ $L^p(M)$, $\widehat{W}^{k,p}(M)$ and $H^1(M)$ are Banach spaces.\\
$(2)$ The Sobolev embedding theorem holds in $\widehat{W}^{k,p}(M)$, so does it in $H^1(M)$.\\
$(3)$ If each nonempty gluing locus $M_j\cap M_l(j\neq l)$ has codimension at least $2$, then $H^1(M)=\widehat{W}^{1,2}(M)$.
\end{sobolev}

\begin{proof}
The proof of (1),(2) is rather standard. Here we show how to prove (3). For the case of codimension at least $3$, one can simply use the linear interpolation functions to approximate an arbitrary $\widehat{W}^{1,2}(M)$ function. As for codimension $2$, consider a function which is $f_R(r)=\frac{\ln{R}}{\ln{r}}$ on a $2$-dimensional ball of radius $R$, and equals to $1$ outside the ball within a larger bounded domain of $\mathbb{R}^2$. The function is smooth everywhere except at the origin. However, if the origin is on the gluing locus, the function is considered to be smooth by our definition. By straightforward calculations, this family of functions approximate a point jump in $W^{1,2}(\mathbb{R}^2)$ norm as $R\to 0$, which can be modified to approximate a codimension $2$ jump at the gluing locus.
\end{proof}

Define $\Delta_{M}$ pointwise on $M-S$, which is the Laplace-Beltrami operator on the manifold part where the point lies. We define the eigenvalue problem of $\Delta_{M}$ as follows.

\begin{eigenvalue}\label{eigenvalue}
Consider the following equation in weak sense,
$$-\Delta_{M} u=\lambda u,$$ more precisely,
$$\sum_{l=1}^{m}\int_{M_l}\nabla u \cdot \nabla \psi\; dv_l=\lambda \sum_{l=1}^{m}\int_{M_l} u\psi \; dv_l, \;\;\forall \psi\in C^{\infty}(M).$$
If there exists a nontrivial solution $u\in H^1(M)$ for some $\lambda\in\mathbb{R}$, then $\lambda$ is an eigenvalue of $-\Delta_M$ and $u$ is an eigenfunction with respect to the eigenvalue $\lambda$.
\end{eigenvalue}

The same as for manifolds, the spectrum of $-\Delta_M$ is discrete and non-negative. Definition \ref{eigenvalue} is equivalent to the following min-max formula:
$$\lambda_k(M)= \inf_{Q^{k+1}} \sup_{u\in Q^{k+1}-\{0\}} \frac{||\nabla u||_{L^2(M)}^2}{||u||_{L^2(M)}^2},$$
where $Q^{k+1}$ ranges over all $(k+1)$-dimensional subspaces of $H^1(M)$. And the minimizers are the solutions of the Laplacian eigenvalue problem. Actually the min-max formula can be applied to prove the existence of the eigenfunctions with the help of the Sobolev embedding theorem. The boundary condition is Neumann, which coincides with the case for manifolds with boundary. We focus on the case where each nonempty gluing locus has codimension $1$, because for codimension at least $2$, the spectrum of $M$ is essentially the spectra of individual manifold parts due to Lemma \ref{sobolev}(3).

\begin{eigenfunction}\label{eigenfunction}
Assume each nonempty gluing locus has codimension $1$. Then the eigenfunctions of $-\Delta_M$ lie in $C^{\infty}(M)$ and they (after orthonormalization) form an orthonormal basis of $L^2(M)$. The eigenfunctions are subject to the Neumann boundary condition and a Kirchhoff-type condition at the gluing locus:
$$\frac{\partial u}{\partial \vec{n}}|_{\partial M}=0,$$ and
$$\sum\frac{\partial u}{\partial \vec{n}}|_{S}=0,$$
where the sum is over all possible inward normal directions at the gluing loci $S$.
\end{eigenfunction}

\begin{proof}
The smoothness of the eigenfunctions is due to the standard regularity procedure. Without loss of generality, assume each gluing locus is connected. We apply the test function being the usual choice inside one manifold part, say $M_1$, and cut off in other manifold parts to satisfy the continuity condition. Recall that we denote by $u_l$ the restriction of $u$ onto $M_l$. The regularity procedure implies $u_1\in C^{\infty}(\overline{M_1-S})$. The same procedure yields $u_l\in C^{\infty}(\overline{M_l-S})$ for all $l$. And $u_l$ must agree at $S$, since a function with a codimension $1$ jump fails to be a $H^1(M)$ function. The boundary condition and the condition at the gluing locus directly follow from integrating by parts.
\end{proof}

Let me give a few examples to show what the eigenfunctions look like on the metric-measure spaces in question.

\begin{circle11}
Consider two identical circles of perimeter $1$ glued at one point. One can think of this space as the interval $[0,2]$, with $0,1$ and $2$ glued. The lower eigenvalues and respective eigenfunctions of the Laplacian eigenvalue problem are as follows.\\
$\lambda_0=0$, $f_0(x)=const$;\\
$\lambda_1=\pi^2$, $f_1(x)=\sin{\pi x}$;\\
$\lambda_2=4\pi^2$, $f_{2,1}(x)=\cos{2\pi x}, f_{2,2}(x)=\sin{2\pi x}$,
$f_{2,3}(x) = \left\{ \begin{array}{ll}
\sin{2\pi x} & x\in [0,1]\\
0 & x\in (1,2]\\
\end{array}. \right.$
Note that $f_{2,3}$ is not smooth on $[0,2]$ with respect to the usual topology, but it is smooth in this metric-measure space by our definition.
\end{circle11}

\begin{circle21}
Consider two circles of perimeter $2,1$ glued at one point. Then the lower eigenvalues and eigenfunctions are:\\
$\lambda_0=0$, $f_0(x)=const$;\\
$\lambda_1=\frac{4}{9}\pi^2$, $f_1(x)=\cos{\frac{2}{3}\pi x}-\sqrt{3}\sin{\frac{2}{3}\pi x}$;\\
$\lambda_2=\pi^2$, $f_{2}(x)=\left\{ \begin{array}{ll}
\sin{\pi x} & x\in [0,2]\\
0 & x\in (2,3]\\
\end{array} \right.$;\\
$\lambda_3=\frac{16}{9}\pi^2$, $\lambda_4=4\pi^2$, with multiplicity $1$ and $3$.
\end{circle21}


\begin{torus11}
Consider two flat tori glued at one point. Since codimension $2$ jumps are permitted, all the eigenfunctions on this space are the combinations of the eigenfunctions on each torus. Hence we see double multiplicities in this case.\\
$\lambda_0=0$, with multiplicity $2$;\\
$\lambda_1=4\pi^2$, with multiplicity $8$;\\
$\lambda_2=8\pi^2$, with multiplicity $8$.
\end{torus11}

\begin{circlesegment}
Consider a circle of perimeter $1$ glued with a segment of length $1$. It is essentially the space of $[0,2]$, with $1$ and $2$ glued.\\
$\lambda_0=0$, $f_0(x)=const$;\\
$\sqrt{\lambda_1}=2\arccos(\frac{\sqrt{3}}{3})$, $f_1(x)=\left\{ \begin{array}{ll}
\cos{\sqrt{\lambda_1} x} & x\in [0,1]\\
-\frac{\sqrt{3}}{3}\cos{\sqrt{\lambda_1}(x-1.5)} & x\in (1,2]\\
\end{array}; \right.$ \\
$\sqrt{\lambda_2}=2\pi-2\arccos(\frac{\sqrt{3}}{3})$, $f_2(x)=\left\{ \begin{array}{ll}
\cos{\sqrt{\lambda_2} x} & x\in [0,1]\\
\frac{\sqrt{3}}{3}\cos{\sqrt{\lambda_2}(x-1.5)} & x\in (1,2]\\
\end{array}; \right.$ \\
$\lambda_3=4\pi^2$, $f_{3,1}(x)=\cos{2\pi x}$, $f_{3,2}(x)=\left\{ \begin{array}{ll}
0 & x\in [0,1]\\
\sin{2\pi x} & x\in (1,2]\\
\end{array}. \right.$ \\
Note that $f_1$ and $f_2$ are not smooth on $[0,2]$ with respect to the usual topology. The derivative splits at point $1$, with half propagating past $1$ and the other half propagating past $2$ in the other direction.
\end{circlesegment}

\section{Proof of Theorem \ref{main2} and \ref{main3}}

In this section, we prove Theorem \ref{main2} and \ref{main3}. Suppose $M$ is a metric-measure space which is glued out of $m$ number of $n$-dimensional compact Riemannian manifolds $M_l\in \mathcal{M}_n(K_1,K_2,D,i_0)$, as defined at the beginning of Section $4$. The gluing loci is denoted by $S$. Suppose the Assumption is satisfied, and the mirror images (Definition \ref{mirror}) are defined near the gluing loci. From now on, we denote by $o(1)$ the rate of $||\nabla \Phi||_r$ converging to $1$ as $r\to 0$ guaranteed by the Assumption. We construct weighted graphs $\Gamma_{\varepsilon,\rho}=\Gamma(X_{\varepsilon},\mu,\rho)$ for a metric-measure space $M$ by choosing weighted graphs with the same parameters $\varepsilon\ll\rho$ for each manifold part $M_l$. The graph Laplacian is defined as
\begin{equation*}
\Delta_{\Gamma}u(x_i) = \frac{2(n+2)}{\nu_n \rho^{n+2}}\sum_l \bigg(\sum_{j: d_l(x_i^l,x_j)<\rho}\mu_j(u(x_j)-u(x_i))+\sum_{j: \widetilde{d_l}(x_i^l,x_j)<\rho}\mu_j(u(x_j)-u(x_i))\bigg),
\end{equation*}
where $x_i^l$ denotes the mirror image of the vertex $x_i$ in $M_l$, and $d_l$ is the Riemannian distance of $M_l$. Its discrete energy is given by
\begin{eqnarray*}
||\delta u||^2 = \frac{n+2}{\nu_n \rho^{n+2}}\sum_{i,l} \big( \sum_{j: d_l(x_i^l,x_j)<\rho}+\sum_{j: \widetilde{d}_l(x_i^l,x_j)<\rho}\big) \mu_i\mu_j |u(x_j)-u(x_i)|^2. \nonumber
\end{eqnarray*}
The graph Laplacian $-\Delta_{\Gamma}$ is a self-adjoint non-negative operator with respect to the inner product in $L^2(X_{\varepsilon})$. The $k$-th eigenvalue of the graph Laplacian is denoted by $\lambda_k(\Gamma)$ and $\lambda_0(\Gamma)=0$. The standard min-max principle for eigenvalues also applies in this case. We prove Theorem \ref{main2} and \ref{main3} by obtaining the lower and upper bound for $\lambda_k(\Gamma)$ in terms of the Laplacian spectrum defined in Section $4$. The proof constantly contains multiple summations, and it is convenient for us to introduce a few notations to simplify reading.

For $x\in M_l$, we define $\widetilde{d}_l$ with respect to the Riemannian distance $d_l$ of $M_l$ as in Definition \ref{geodesic}, and define $\widehat{d}_l$ as
\begin{equation}\label{dl}
\{y:\widehat{d}_l(x,y)<r\}=\{y:d_l(x,y)<r\}\sqcup \{y:\widetilde{d}_l(x,y)<r\},
\end{equation}
where the union is a disjoint union. Denote the two sets on the right-hand side above by $B_r(x)$ and $\widetilde{B}_r(x)$. And we define $\widehat{d}$ (without an index) as follows:
\begin{equation}\label{d}
\{y:\widehat{d}(x,y)<r\}=\bigsqcup_l \{y:\widehat{d}_l(x^l,y)<r\}.
\end{equation}
Define the disjoint union of all reflected balls centered at the mirror images by
\begin{equation}\label{Bstar}
\widehat{B}_r^{\ast}(x)=\bigsqcup_{l} \widehat{B}_r(x^l),
\end{equation}
where the reflected ball $\widehat{B}_r(x^l)$ is the reduced domain where only the classical reflection at the boundary can occur, as in Definition \ref{ball}. Thus $\widehat{B}_r^{\ast}(x)$ by definition also only involves the classical reflection, while $\{y:\widehat{d}(x,y)<r\}$ involves all possible behaviors of geodesics. In this way, the discrete energy can be conveniently written as
\begin{equation}\label{deltau}
||\delta u||^2=\frac{n+2}{\nu_n \rho^{n+2}}\sum_{i} \sum_{j: \widehat{d} (x_i,x_j)<\rho}\mu_i\mu_j |u(x_j)-u(x_i)|^2.\\
\end{equation}

\subsection*{Lower bound for $\lambda_k(\Gamma)$}

The main difficulty is to obtain the lower bound for $\lambda_k(\Gamma)$, which demands a suitable kernel to do the job as explained in Section $1$. Without loss of generality, assume each gluing locus is connected. For an arbitrary point $x\in M$, say $x\in M_j$, we define mirror images $x^l$ of $x$ on all $M_l$ satisfying $M_j\cap M_l\neq\emptyset$. For a fixed $r$ satisfying $r<\rho\ll 1$ and a given function $f\in L^2(M)$, we define a Lipschitz function $\Lambda_r f \in H^1(M)$ by
$$\Lambda_r f(x)=\int_M k_r(x,y)f(y)dy,$$
where
$$k_r(x,y)=\frac{1}{1+\sum_{l\neq j}\alpha_l(x)}\frac{1}{r^n}\bigg(\phi(\frac{\widehat{d_j}(x,y)}{r})\chi_{\widehat{B}_r(x)}(y)+\sum_{l\neq j}\alpha_l(x)
\phi(\frac{\widehat{d_l}(x^l,y)}{r})\chi_{\widehat{B}_r(x^l)}(y)\bigg),$$
where the sum ranges over all $l$ satisfying $M_j\cap M_l\neq\emptyset$, and $\chi$ is the characteristic function. The function $\phi:\mathbb{R}_{\geq 0}\to \mathbb{R}_{\geq 0}$ is defined as $\phi(t)=\frac{n+2}{2\nu_n}(1-t^2)$ if $t\in[0,1]$ otherwise $0$, and
$$\alpha_l(x)=\left\{ \begin{array}{lll}
0,\;\;\;d_j(x,M_j\cap M_l)\geqslant r^{\frac{3}{4}};\\
1,\;\;\;x\in M_j\cap M_l;\\
1-\frac{1}{r^{\frac{3}{4}}}d_j(x,M_j\cap M_l),\;\; \textrm{otherwise}.\\
\end{array} \right.$$

Note that $k_r(x,y)$ is not continuous, but its integral with respect to $y$ is continuous in $x$, since the reflected ball $\widehat{B}_r(x)$ varies continuously. From the definition of $\alpha_l$, we immediately have $|\nabla\alpha_l|\leqslant\frac{C}{r^{\frac{3}{4}}}$ almost everywhere, where the gradient is taken with respect to the Riemannian structure of $M_j$. This is due to the fact that the function $h(x)=d_j(x, M_j\cap M_l)$ is a Lipschitz function with the Lipschitz constant $1$. Therefore it is differentiable almost everywhere and $|\nabla_x d_j(x, M_j\cap M_l)|\leqslant \sqrt{n}$. The continuity of $\Lambda_r f$ at the gluing locus is guaranteed by the Assumption, the definition of mirror images and the cut-off function $\alpha_l(x)$. To ensure continuity at the exact $\rho^{\frac{3}{4}}$-distance from the gluing locus, the distance within which $\alpha_l(x)$ is nonzero has to be smaller than the distance within which the mirror images are defined, which means exactly $r<\rho$.\\
\indent Let's consider for a moment a simple case: three book pages glued along an edge. For a point sufficiently close to the edge on the first plane, we use the symmetry Assumption to define its mirror images on the other two planes. Now if we choose a point far from the edge, the point should not be affected by other planes at all, and this is reflected by the cut-off function $\alpha_l$ being zero, which results in only information from the first plane being gathered by the kernel $k_r$. As the chosen point gets closer to the edge, the function $\alpha_l$ starts being nonzero, causing the kernel $k_r$ to gather information from two other planes, and their impacts become larger as $\alpha_l$ grows. Finally, when the chosen point reaches the edge, $\alpha_l$ achieves $1$, which means all three planes will have equal impact on the kernel, since the mirror images also move to the same point on the edge thanks to the Assumption. This means if we travel on any plane to that same point on the edge, the information being gathered will be exactly the same, which ensures the continuity at the gluing locus.\\
\indent Define $\theta(x)=\Lambda_r (1_M)$ for $r<\rho\ll 1$, and we start with an estimate of the value and the derivative of $\theta(x)$.

\begin{newtheta}\label{newtheta}
For almost every $x\in M$, we have\\
$(1)$ $|\theta(x)-1|\leqslant Cr$;\\
$(2)$ $|\nabla\theta(x)|\leqslant \frac{C}{r}o(1)+\frac{C}{r^{\frac{3}{4}}}(1+r)$,
where $o(1)$ is the rate of $||\nabla\Phi||_r$ converging to $1$ as $r\to 0$ in the Assumption and the constants explicitly depend on $n,m,K_1,K_2$.\\
In particular, if the assumption of Theorem \ref{main3} is satisfied, the second inequality improves to\\
$(2^{\ast})$ $|\nabla\theta(x)|\leqslant \frac{C}{r^{\frac{3}{4}}}(1+r)$.
\end{newtheta}

\begin{proof}
(1) is straightforward considering the fact that $\frac{1}{r^n}\int_{\mathbb{R}^n}\phi(\frac{d(x,y)}{r})dy=1$ for $x\in \mathbb{R}^n$, and the factor $(1+Cr)$ comes from the Jacobian estimate in Lemma \ref{Jacobi}. One can refer to a similar argument in Lemma \ref{thetax}. \\
\indent For $(2)$, given a point $x\in M_j$, we have
$$\frac{\partial}{\partial x}d_l (x^l,y)=-(\nabla\Phi_{jl})(x)\frac{exp_{x^l}^{-1}(y)}{d_l (x^l,y)}.$$
By a straightforward calculation we get
$$\nabla \theta(x)=\int_M \frac{\partial}{\partial x}k_r (x,y) dy,$$
and
\begin{eqnarray}\label{kr}
\frac{\partial}{\partial x}k_r (x,y) &=& \frac{1}{1+\sum_l\alpha_l}\frac{n+2}{\nu_n r^{n+2}}exp_x^{-1}(y)\chi_{\widehat{B}_r(x)} \nonumber \\
&+& \frac{1}{1+\sum_l \alpha_l}\frac{n+2}{\nu_n r^{n+2}} \sum_{l}\alpha_l(\nabla\Phi) exp_{x^l}^{-1}(y)\chi_{\widehat{B}_r(x^l)} \nonumber \\
&+& \frac{1}{1+\sum_l\alpha_l}\frac{n+2}{2\nu_n r^n}\sum_{l}(\nabla\alpha_l)\phi(\frac{\widehat{d}_l(x^l,y)}{r})\chi_{\widehat{B}_r(x^l)}\nonumber \\
&+& (\nabla\frac{1}{1+\sum_l\alpha_l})\frac{1}{r^n}\bigg(\phi(\frac{\widehat{d}_j(x,y)}{r})\chi_{\widehat{B}_r(x)}+\sum_{l}\alpha_l\phi(\frac{\widehat{d}_l(x^l,y)}{r})\chi_{\widehat{B_r}(x^l)}\bigg). \nonumber \\
\end{eqnarray}
Note that the sum ranges over all $l\neq j$ satisfying $M_j\cap M_l\neq\emptyset$. Once integrating (\ref{kr}) over $M$, the last two terms are bounded by $|\nabla\alpha_l|\leqslant C/r^{3/4}$, keeping in mind that $\frac{1}{r^n}\int_{\mathbb{R}^n}\phi(\frac{d(x,y)}{r})dy=1$ for $x\in \mathbb{R}^n$.\\
\indent Now let's deal with the first two terms of (\ref{kr}). In a similar estimate for manifolds (see Lemma \ref{thetax}), we replaced the Jacobian $J_x$ by $J_x-1$ in order to improve the rate, thanks to the symmetry of $\mathcal{B}_r(0)\subset T_x M$ with respect to the origin and the fact that $\mathcal{W}_r(x)$ has small measure. This gives us some extra power of $r$ to compensate for the denominator $r^{n+2}$. We still have the same thing for the first term. But for the second term, the symmetry of $\mathcal{B}_r(0)$ may be distorted under the map $\nabla\Phi$. However, thanks to the Assumption that $\nabla \Phi$ almost preserves length near the gluing locus, after canceling out the vectors of opposite directions, all the resulting vectors are contained in a ball of radius $r o(1)$, where $o(1)$ is the rate of $||\nabla\Phi||_r$ converging to $1$ as $r\to 0$. Hence the second inequality follows, which yields $(2^{\ast})$ with $o(1)=C(K_1)r^2$ in that specific case from previous discussions in Section $1$.
\end{proof}
Note that the extra rate of $r$ we get from the symmetry Assumption is crucial. We will see the proof does not work without the extra rate in a similar way as the case for manifolds. \\
\indent For a function $f\in L^2(M)$, we consider the Lipschitz function $I_r f:=\theta^{-1}\Lambda_r f$. We set $r=\rho-4\varepsilon$ and $f=P^{\ast}u$, where $P^{\ast}u$ is the piecewise constant function defined in Remark \ref{remark} for a discrete function on the $\varepsilon$-net $X_{\varepsilon}$. Denote this specific choice by $Iu:=I_{\rho-4\varepsilon}(P^{\ast}u)=\theta^{-1}\Lambda_{\rho-4\varepsilon}(P^{\ast}u)$. We define the analogue of the discrete energy by
\begin{equation}\label{Er}
\widehat{E}_r(f)=\int_M\int_{\widehat{B}_r^{\ast}(x)}|f(y)-f(x)|^2 dy dx,
\end{equation}
where $\widehat{B}_r^{\ast}(x)$ is defined in (\ref{Bstar}). We proceed to control the $L^2$-norm and energy of $I u$ in terms of $\widehat{E}_{\rho-4\varepsilon} (P^{\ast}u)$ in the following lemma.
\begin{new2}\label{new2}
For $u\in L^2(X_{\varepsilon})$, we have the following estimates:
\begin{flalign*}
&(1)\;\;\; ||Iu-P^{\ast}u||_{L^2(M)}^2\leqslant \frac{C}{(\rho-4\varepsilon)^n}\widehat{E}_{\rho-4\varepsilon}(P^{\ast}u); \\
&(2)\;\;\; ||\nabla(Iu)||_{L^2(M)}^2 \leqslant \frac{n+2}{\nu_n (\rho-4\varepsilon)^{n+2}}(1+o(1)+C\rho^{\frac{1}{4}})\widehat{E}_{\rho-4\varepsilon} (P^{\ast}u),
\end{flalign*}
where $\widehat{E}_r(f)$ is defined in (\ref{Er}), and the constants explicitly depend on $n,m,K_1, K_2$.
\end{new2}

\begin{proof}
We prove the lemma for an arbitrary choice of $r$ and $f\in L^2(M)$. First we estimate the bound of the $L^2$-norm. By the definition of $I_r$ and $\theta$, we have
\begin{eqnarray}\label{Irf}
I_r f(x)-f(x) &=& \theta^{-1}(\Lambda_r f)(x)-f(x) \nonumber \\
& = & \theta^{-1}(\Lambda_r f)(x)- \theta^{-1}\Lambda_r (f(x)1_M) \nonumber \\
&=& \theta^{-1}\int_{M}(f(y)-f(x)) k_r(x,y) dy.
\end{eqnarray}
Since $|\theta|\leqslant C$, $k_r\leqslant \frac{C}{r^n}$ and $k_r$ is only supported in $\widehat{B}_r^{\ast}(x)$ from the definition of $k_r$, by the Cauchy-Schwarz inequality we get
\begin{eqnarray*}
|I_r f(x)-f(x)|^2 &=& |\theta^{-1}\int_M (f(y)-f(x))k_r(x,y)dy|^2\\
&\leqslant& \theta^{-2}(\int_M k_r(x,y) dy)(\int_M |f(y)-f(x)|^2 k_r(x,y)dy)\\
&\leqslant& \frac{C}{r^n}\int_{\widehat{B}_r^{\ast}(x)}|f(y)-f(x)|^2 dy.
\end{eqnarray*}
Integrate over $M$, and we obtain
\begin{equation}
||I_r f-f||_{L^2(M)}^2\leqslant \frac{C}{r^n}\widehat{E}_r(f).
\end{equation}
This finishes the first part of the lemma. \\

\indent Next we turn to the energy estimate. Differentiating (\ref{Irf}) at a point $x\in M_j-S$, and by (\ref{kr}) we have
\begin{eqnarray*}
\nabla(\theta^{-1}\Lambda_r f)(x) &=& \theta^{-1}\int_{M} (f(y)-f(x))\frac{\partial}{\partial x}k_r(x,y) dy + \nabla(\theta^{-1})\int_{M} (f(y)-f(x))k_r(x,y) dy \\
&=& \frac{1}{1+\sum\alpha_l}\frac{n+2}{2\nu_n r^n}\theta^{-1}\sum_l(\nabla\alpha_l)\int_{\widehat{B}_r(x^l)}(f(y)-f(x))\phi(\frac{\widehat{d}_l(x^l,y)}{r})dy\\
&+& (\nabla\frac{1}{1+\sum\alpha_l})(1+\sum_l\alpha_l)\theta^{-1}\int_M (f(y)-f(x))k_r(x,y)dy \\
&+& \nabla(\theta^{-1})\int_{M} (f(y)-f(x))k_r(x,y)dy + \theta^{-1} A_1,
\end{eqnarray*}
where the second term is obtained by the definition of $k_r$, and 
\begin{eqnarray*}
A_1(x)&=&\frac{1}{1+\sum\alpha_l}\frac{n+2}{\nu_n r^{n+2}}\bigg(\int_{\widehat{B}_r(x)} (f(y)-f(x))exp_x^{-1}(y)dy \\
&&+ \sum_{l}\alpha_l \int_{\widehat{B}_r(x^l)} (f(y)-f(x))(\nabla\Phi)exp_{x^l}^{-1}(y)dy \bigg).
\end{eqnarray*}
We denote the first three terms by $A_2$, $A_3$, and $A_4$ respectively. Note that the sum ranges over all $l\neq j$ satisfying $M_j\cap M_l\neq\emptyset$. Due to $|\phi|\leqslant 1$, $|\nabla\alpha_l|\leqslant C/r^{3/4}$ and the Cauchy-Schwarz inequality, we get
\begin{eqnarray*}
|A_2|^2 &=& \frac{1}{(1+\sum\alpha_l)^2}\frac{(n+2)^2}{4\nu^2_n r^{2n}}\theta^{-2}|\sum_{l}(\nabla\alpha_l)\int_{\widehat{B}_r(x^l)}(f(y)-f(x))\phi(\frac{\widehat{d}_l(x^l,y)}{r})dy|^2 \\
&\leqslant& \frac{C}{r^{2n}}\sum_l |\nabla\alpha_l|^2\big(\sum_l \int_{\widehat{B}_l(x^l)}|f(y)-f(x)|^2 dy\int_{\widehat{B}_r(x^l)} 1 dy\big)\\
&\leqslant& \frac{C}{r^{n+\frac{3}{2}}}\int_{\widehat{B}_r^{\ast}(x)}|f(y)-f(x)|^2 dy.
\end{eqnarray*}
The last inequality is due to the volume estimate in Lemma \ref{Jacobi}. The same bound also applies to $A_3$. For $A_4$, since $|\nabla\theta|\leqslant \frac{C}{r}o(1)+\frac{C}{r^{\frac{3}{4}}}(1+r)$ by Lemma \ref{newtheta}, by the Cauchy-Schwartz inequality we have
\begin{eqnarray*}
|A_4|^2 &\leqslant& \frac{C}{r^{2}}(o(1)+r^{\frac{1}{2}}) \big(\int_M |f(y)-f(x)|k_r(x,y) dy\big)^2 \\
&\leqslant& \frac{C}{r^{2}}(o(1)+r^{\frac{1}{2}}) \big(\int_M |f(y)-f(x)|^2 k_r(x,y) dy\big)\big(\int_M k_r(x,y)dy\big) \\
&\leqslant& \frac{C}{r^{2}}(o(1)+r^{\frac{1}{2}}) \int_M |f(y)-f(x)|^2 k_r(x,y) dy,
\end{eqnarray*}
where the last inequality is due to the definition and the boundedness of $\theta$. Since $k_r$ is only supported in $\widehat{B}_r^{\ast}(x)$ and $k_r\leqslant \frac{C}{r^n}$, we get
$$|A_4|^2 \leqslant \frac{C}{r^{n+2}}(o(1)+r^{1/2}) \int_{\widehat{B}_r^{\ast}(x)} |f(y)-f(x)|^2 dy.$$
At last, we come to $A_1$. Denote its universal term by
$$A(x) = \frac{n+2}{\nu_n r^{n+2}}\int_{\widehat{B}_r (x^l)} (f(y)-f(x))(\nabla\Phi)exp_{x^l}^{-1}(y)dy.$$
For $w=\frac{A(x)}{|A(x)|}$, we have $|A(x)|=\langle A(x),w\rangle$. Then we get
$$|A(x)|\leqslant \frac{n+2}{\nu_n r^{n+2}}\int_{\mathcal{B}_r(0)-\mathcal{W}_r(x^l)\subset T_{x^l}M_l}|(f(exp_{x^l}(v))-f(x))\langle (\nabla\Phi)v,w \rangle J_{x^l}(v)| dv.$$
By the Cauchy-Schwarz inequality and (\ref{sublemma}), we have
\begin{eqnarray*}
|A(x)|^2 &\leqslant& (\frac{n+2}{\nu_n r^{n+2}})^2(\int_{\mathcal{B}_r(0)-\mathcal{W}_r(x^l)}|f(exp_{x^l}(v))-f(x)|^2 |J_{x^l}(v)|^2 dv)(\int_{\mathcal{B}_r(0)} \langle (\nabla\Phi)v,w\rangle^2 dv) \\
&\leqslant& (\frac{n+2}{\nu_n r^{n+2}})^2(\int_{\mathcal{B}_r(0)-\mathcal{W}_r(x^l)}|f(exp_{x^l}(v))-f(x)|^2 |J_{x^l}(v)|^2 dv)(\int_{\mathcal{B}_{r+ro(1)}(0)} \langle v,w\rangle^2 dv) \\
&=& \frac{n+2}{\nu_n r^{n+2}}(1+o(1))\int_{\mathcal{B}_r(0)-\mathcal{W}_r(x^l)}|f(exp_{x^l}(v))-f(x)|^2 |J_{x^l}(v)|^2 dv \\
&\leqslant& \frac{n+2}{\nu_n r^{n+2}}(1+o(1)+Cr)\int_{\widehat{B}_r^{\ast}(x)}|f(y)-f(x)|^2 dy,
\end{eqnarray*}
where we used $||\nabla\Phi||_r \leqslant 1+o(1)$, which causes the radius of the ball to enlarge by a factor $1+o(1)$, and the Jacobian estimate (Lemma \ref{Jacobi}). Therefore, we get
$$|A(x)|\leqslant \sqrt{\frac{n+2}{\nu_n r^{n+2}}}(1+o(1)+Cr)\sqrt{\int_{\widehat{B}_r^{\ast}(x)}|f(y)-f(x)|^2 dy}.$$
Observe that every term in $A_1$ can be bounded in the same way as $A(x)$. After adding all the factors $\alpha_l(x)$ and then averaging by $1+\sum\alpha_l(x)$, we get exact the same estimate for $A_1(x)$.
$$|A_1(x)|\leqslant \sqrt{\frac{n+2}{\nu_n r^{n+2}}}(1+o(1)+Cr)\sqrt{\int_{\widehat{B}_r^{\ast}(x)}|f(y)-f(x)|^2 dy}.$$
Now combine the estimates for $A_2,\,A_3,\,A_4$, and we obtain
\begin{eqnarray*}
|\nabla(\theta^{-1}\Lambda_r f)(x)| &\leqslant& |A_1|+|A_2|+|A_3|+|A_4| \\
&\leqslant& \sqrt{\frac{n+2}{\nu_n r^{n+2}}}(1+o(1)+Cr^{\frac{1}{4}})\sqrt{\int_{\widehat{B}_r^{\ast}(x)}|f(y)-f(x)|^2 dy}.
\end{eqnarray*}
Therefore, we finally obtain
\begin{equation}
||\nabla(I_r f)||_{L^2(M)}^2 \leqslant \frac{n+2}{\nu_n r^{n+2}}(1+o(1)+Cr^{\frac{1}{4}})\widehat{E}_r (f).
\end{equation}
\end{proof}

\subsection*{Upper bound for $\lambda_k(\Gamma)$}

Now we turn to the upper bound for $\lambda_k(\Gamma)$. Given a weighted graph $\Gamma_{\varepsilon,\rho}=\Gamma(X_{\varepsilon},\mu,\rho)$, define the discretization operator $P:L^2(M)\to L^2(X)$ by
$$Pf(x_i)=\mu_i^{-1}\int_{V_i} f(x)dx,$$
and $P^{\ast}: L^2(X)\to L^2(M)$ by
$$P^{\ast}u=\sum_{i=1}^{N}u(x_i)1_{V_i},$$
where $N$ is the number of points of the $\varepsilon$-net $X_{\varepsilon}$. It immediately follows that $||P^{\ast}u||_{L^2 (M)}=||u||_{L^2 (X)}$. Lemma \ref{2.3} stays true since the weighted graph is constructed by choosing weighted graphs for each manifold part, so we only need to obtain the energy estimate. Denote by $\Omega=\Omega_{2\rho^\frac{3}{4}}$ the region within $2\rho^{\frac{3}{4}}$ respective Riemannian distance from the gluing loci $S$. Recall that $\rho^{\frac{3}{4}}$ is the distance within which we define mirror images. Outside $\Omega$, everything is the same as for one single manifold. Therefore Lemma \ref{2.2} (or Lemma $3.3$ in \cite{BIK}) holds for this region, namely for $r<2\rho$ and $f\in C^{\infty}(M)$,
\begin{equation}\label{ErOc}
\widehat{E}_{r}(f,\Omega^c)\leqslant (1+Cr)\frac{\nu_n r^{n+2}}{n+2}||\nabla f||_{L^2(M)}^2,
\end{equation}
where $$\widehat{E}_{r}(f,V)=\int_V\int_{\widehat{B}_r^{\ast}(x)}|f(y)-f(x)|^2 dy dx,\textrm{  for }V\subset M.$$
Inside $\Omega$, we can estimate the energy in terms of the $L^{\infty}$-norm $||\nabla f||_{L^{\infty}(M)}$.
\begin{eqnarray}\label{ErO}
\widehat{E}_{r}(f,\Omega) &\leqslant& C\int_{\Omega}\int_{\widehat{B}_r^{\ast}(x)} ||\nabla f||_{\infty}^2 (\rho^{\frac{3}{4}}+r)^2 dy dx \nonumber \\
&\leqslant& C||\nabla f||_{\infty}^2 \rho^{\frac{3}{4}} r^n (\rho^{\frac{3}{4}}+r)^2,
\end{eqnarray}
where we used the facts that the volume $vol(\Omega)\leqslant Cvol_{n-1}(S)\rho^{\frac{3}{4}}$, and the volume $vol(\widehat{B}_r^{\ast}(x)) \leqslant C(1+o(1))r^n$ (due to the Assumption), and within $\Omega$ the distance between $x$ and its mirror image $x^l$ is controlled by $\rho^{\frac{3}{4}}(2+o(1))$ (again due to the Assumption). Here $vol_{n-1}(S)$ denotes the $(n-1)$-dimensional volume of the gluing locus $S$. Combining all these, we prove the following lemma.
\begin{new3}\label{new3}
For $f\in C^{\infty} (M)$, we have the following estimates:
\begin{flalign*}
&(1)\;\;\; ||f-P^{\ast}P f||_{L^2}^2\leqslant C\rho^{2}(||\nabla f||_{L^2(M)}^2+||\nabla f||_{L^{\infty}(M)}^2);\\
&(2)\;\;\; \widehat{E}_{\rho+4\varepsilon}(f,\Omega^c)\leqslant (1+C\rho)\frac{\nu_n (\rho+4\varepsilon)^{n+2}}{n+2}||\nabla f||_{L^2(M)}^2;\\
&(3)\;\;\; \widehat{E}_{\rho+4\varepsilon}(f,\Omega)\leqslant C||\nabla f||_{\infty}^2 (1+o(1))(\rho+4\varepsilon)^n\rho^{\frac{9}{4}},
\end{flalign*}
where $\Omega=\Omega_{2\rho^\frac{3}{4}}$ is the region within $2\rho^{\frac{3}{4}}$ respective Riemannian distance from the gluing loci $S$, and the constants explicitly depend on $n,m,K_1,K_2,vol(M),vol_{n-1}(S)$.
\end{new3}

\begin{proof}
The last two inequalities were already proved by setting $r=\rho+4\varepsilon$ in (\ref{ErOc}) and (\ref{ErO}). By Lemma \ref{2.3}, we get,
$$\int_{V_i} |f(x)-P f(x_i)|^2 dx \leqslant \frac{C}{r^n} \int_{V_i}\int_{\{y:\widehat{d}(x,y)<r\}} |f(y)-f(x)|^2 dy dx.$$
We emphasize that the difference between $\widehat{B}_r^{\ast}(x)$ and $\{y:\widehat{d}(x,y)<r\}$ in (\ref{d}), (\ref{Bstar}). The former contains only the simple behavior of the classical reflection at the boundary with no geodesics touching the boundary, while the latter contains all possible behaviors. Then by the definition of $P^{\ast}$, (\ref{ErOc}) and (\ref{ErO}), we get
\begin{eqnarray*}
||f-P^{\ast}P f||_{L^2}^2 &=& \sum_i \int_{V_i} |f(x)-P f(x_i)|^2 dx \\
&\leqslant& \frac{C}{r^n} \bigg(\widehat{E}_{r}(f) +  2\sum_{l}\int_{M}\int_{exp(\mathcal{W}_r(x^l))} |f(y)-f(x)|^2 dy dx\bigg)\\
&\leqslant &  \frac{C}{r^n}(\widehat{E}_{r}(f,\Omega^c)+\widehat{E}_{r}(f,\Omega)) + \frac{C}{r^n}m(\mathcal{W}_r(x^l)) (\rho^{\frac{3}{4}}+r)^2 ||\nabla f||_{L^{\infty}(M)}^2 \\
&\leqslant & Cr^2||\nabla f||_{L^2(M)}^2+C\rho^{\frac{9}{4}}||\nabla f||_{\infty}^2 + Cr(\rho^{\frac{3}{4}}+r)^2||\nabla f||_{L^{\infty}(M)}^2. 
\end{eqnarray*}
The second last inequality is due to the fact that the distance between $x$ and its mirror image $x^l$ (if any) is controlled by $\rho^{\frac{3}{4}}(2+o(1))$ due to the Assumption.
We need to add $\mathcal{W}_r(x)$ back twice for each manifold part, because both $d$ and $\widetilde{d}$ include points which are reached by geodesics touching the boundary. The first inequality follows by setting $r=\rho$.
\end{proof}
So far we have obtained estimates in terms of $\widehat{E}_r(f)$, and the next lemma states the relation between $\widehat{E}_r(f)$ and the discrete energy (\ref{deltau}).

\begin{new4}\label{new4}
For $f\in L^2(M)$ and $u\in L^2(X)$, we have the following estimates:
\begin{flalign*}
&(1)\;\;\; ||\delta u||^2\geqslant \frac{n+2}{\nu_n \rho^{n+2}}\widehat{E}_{\rho-4\varepsilon}(P^{\ast}u);\\
&(2)\;\;\; ||\delta(Pf)||^2\leqslant \frac{n+2}{\nu_n \rho^{n+2}}\widehat{E}_{\rho+4\varepsilon}(f)+C(n,m,vol(M))\rho^{\frac{1}{2}}||\nabla f||^2_{L^{\infty}(M)}.
\end{flalign*}
\end{new4}

\begin{proof}
Since $\widetilde{d}$ satisfies the triangle inequality, we have
\begin{equation}\label{triangle}
\{y: \widehat{d}(x_i,y)<\rho-4\varepsilon\} \subset \bigcup_{j: \widehat{d}(x_i,x_j)<\rho} V_j \subset \{y: \widehat{d}(x_i,y)<\rho+4\varepsilon\},
\end{equation}
which implies the first inequality, since the discrete energy already contains more information than $\widehat{E}_r$. To obtain the second one, we need to add back $\mathcal{W}_r$ (twice) which are the situations of geodesics touching the boundary or geodesics intersecting the boundary elsewhere with the presence of one collision point, as defined in Definition \ref{defW}. Thanks to Lemma \ref{W} that $\mathcal{W}_r$ has small measure, this only generates higher order terms. One can refer to a similar estimate in (\ref{Edelta3}).
\end{proof}

Now we are in place to prove Theorem $\ref{main2}$ and $\ref{main3}$.

\begin{proof}[Proof of Theorem \ref{main2} and \ref{main3}]
We only prove the upper bound for $\lambda_k(\Gamma)$; the lower bound can be proved with the same method. We choose the first $k+1$ eigenfunctions $f_0,\cdots,f_{k}$ of $-\Delta_M$ with respect to the eigenvalues $\lambda_0,\lambda_1,\cdots,\lambda_{k}$. We consider the linear space spanned by $Pf_0,\cdots,Pf_{k}$ which is a subspace of $L^2(X)$, and we first prove this subspace is $(k+1)$-dimensional for sufficiently small $\rho$.\\
\indent On $M_l-S$ for any $l$, the eigenfunction $f_k$ is smooth and satisfies the eigenvalue equation pointwise by Proposition \ref{eigenfunction}. Then by (\ref{morrey0}) and the fact that the volume of any $M_l$ is bounded by a constant depending on $n,D,K_1$, we have
$$||\nabla f_k||_{L^{\infty}(M_l)}^2 \leqslant C(n,i_0,D,K_1)(\lambda_k^{n/2+2}+1)||f_k||^2_{L^2(M)},$$
which implies 
\begin{equation}\label{morrey}
||\nabla f_k||_{L^{\infty}(M)}^2 \leqslant C(n,i_0,D,K_1)(\lambda_k^{n/2+2}+1)||f_k||^2_{L^2(M)}.
\end{equation}
Then by Lemma \ref{new3}(1) and (\ref{morrey}), for any $f$ in the first $k+1$ eigenspaces we get
\begin{equation}\label{injective}
||Pf||_{L^2(X)} \geqslant \big(1-C\rho\sqrt{\lambda_k^{n/2+2}+1}\,\big)||f||_{L^2(M)}.
\end{equation}
Hence for sufficiently small $\rho$ depending on $n,m,i_0,D,K_1,\lambda_k$, the linear operator $P$ is injective, which implies that the linear subspace spanned by $Pf_0,\cdots,Pf_{k}$ is $(k+1)$-dimensional.\\
\indent For any $f\in span\{f_0,\cdots,f_{k}\}$, by Lemma \ref{new3}, \ref{new4} and (\ref{morrey}) we get
\begin{eqnarray*}
||\delta(Pf)||^2 &\leqslant& \frac{n+2}{\nu_n \rho^{n+2}}\widehat{E}_{\rho+4\varepsilon}(f)+ C\rho^{\frac{1}{2}}||\nabla f||^2_{L^{\infty}(M)} \\
&\leqslant& (1+C\rho)\frac{ (\rho+4\varepsilon)^{n+2}}{\rho^{n+2}}||\nabla f||_{L^2(M)}^2 + C||\nabla f||_{\infty}^2 (1+o(1))\frac{(\rho+4\varepsilon)^n}{\rho^{n-\frac{1}{4}}} \\
&\leqslant& (1+C\rho+C\frac{\varepsilon}{\rho})||\nabla f||^2_{L^2(M)}+C\rho^{\frac{1}{4}}(\lambda_k^{n/2+2}+1)||f||^2_{L^2(M)}.
\end{eqnarray*}
Combine this inequality with (\ref{injective}), and for sufficiently small $\varepsilon, \rho$ we obtain
$$\frac{||\delta(Pf)||^2}{||Pf||^2_{L^2}} \leqslant (1+C\rho+C\frac{\varepsilon}{\rho})\lambda_k + C\rho^{\frac{1}{4}}(\lambda_k^{n/2+2}+1),$$
which implies the upper bound for $\lambda_k(\Gamma)$ due to the min-max principle. The same procedure with Lemma \ref{new2} yields the lower bound for $\lambda_k(\Gamma)$. Finally we obtain
\begin{equation}\label{rate}
|\lambda_k(\Gamma)-\lambda_k|\leqslant C(\rho^{\frac{1}{4}}+\frac{\varepsilon}{\rho}+\rho^{\frac{1}{4}}\lambda_k^{\frac{n}{2}+1}+o(1))\lambda_k + C\rho^{\frac{1}{4}},
\end{equation}
where $o(1)$ denotes the rate of $||\nabla\Phi||_{\rho}$ converging to $1$ as $\rho\to 0$, and the constant $C$ explicitly depends on $n,m,i_0,D,K_1,K_2,vol_{n-1}(S)$.\\
\indent In particular, if the assumption of Theorem \ref{main3} is satisfied, we know that $||\nabla\Phi||_{\rho}$ converges to $1$ with an explicit rate $o(1)=C(K_1)\rho^2$. Theorem \ref{main3} immediately follows from (\ref{rate}).\\
\end{proof}


\begin{thebibliography}{99}


\bibitem{AA}
M. Aanjaneya, F. Chazal, D. Chen, M. Glisse, L. Guibas, D. Morozov,
{\it Metric graph reconstruction from noisy data},
Proc. 27th Sympos. Comput. Geom. 2011.

\bibitem{ABB1} S.~Alexander,~I.~Berg,~R.~Bishop,~\emph{Cauchy uniqueness in the Riemannian obstacle problem},~Differential Geometry Pe$\tilde{\textrm{n}}$iscola,~1985.
\bibitem{ABB} S.~Alexander,~I.~Berg,~R.~Bishop,~\emph{The Riemannian obstacle problem},~Illinois J. Math.~\textbf{31}~(1987),~167-184.
\bibitem{ABB2} S.~Alexander,~I.~Berg,~R.~Bishop,~\emph{Geometric curvature bounds in Riemannian manifolds with boundary},~Trans. Amer. Math. Soc.~\textbf{339}~(1993),~703-716.
\bibitem{BN} M.~Belkin,~P.~Niyogi,~\emph{Convergence of Laplacian eigenmaps},~Adv. in NIPS, 2007.

\bibitem{BQWZ}
M. Belkin, Q. Que, Y. Wang, X. Zhou,
{\it Toward understanding complex spaces: graph Laplacians on manifolds with singularities and boundaries},
Proc. 25th Ann. Conf. Learn. Theory, PMLR, 2012.

\bibitem{BIK} D.~Burago,~S.~Ivanov,~Y.~Kurylev,~\emph{A graph discretization of the Laplace-Beltrami operator},~J. Spectr. Theory.~\textbf{4}~(2014),~675-714.
\bibitem{BIK2} D.~Burago,~S.~Ivanov,~Y.~Kurylev,~\emph{Spectral stability of metric-measure Laplacians},~Israel J. Math.~\textbf{232}~(2019),~125-158.


\bibitem{F} K.~Fujiwara,~\emph{Eigenvalues of Laplacians on a closed Riemannian manifold and its nets},~Proc. Amer. Math. Soc. \textbf{123}~(1995),~2585-2594.


\bibitem{LM} G.~Lebeau,~L.~Michel,~\emph{Semi-classical analysis of a random walk on a manifold},~Ann. Probab.~\textbf{38}~(2010),~277-315.

\bibitem{S}
Z. Shi,
{\it Convergence of Laplacian spectra from random samples},
arXiv:1507.00151.

\bibitem{SW}
A. Singer, H. Wu,
{\it Spectral convergence of the connection Laplacian from random samples},
Inf. Inference {\bf 6} (2017), 58-123.

\bibitem{TGHS}
N. Trillos, M. Gerlach, M. Hein, D. Slep$\check{\textrm{c}}$ev,
{\it Error estimates for spectral convergence of the graph Laplacian on random geometric graphs towards the Laplace-Beltrami operator},
to appear in Found. Comput. Math.

\bibitem{TS}
N. Trillos, D. Slep$\check{\textrm{c}}$ev,
{\it A variational approach to the consistency of spectral clustering},
Appl. Comput. Harmon. Anal. {\bf 45} (2018), 239-281.

\bibitem{W}
X. Wang,
{\it Spectral convergence rate of graph Laplacian},
arXiv:1510.08110.



\end{thebibliography}
\end{document}